\documentclass[11pt]{amsart}
\usepackage[usenames]{color}
\usepackage{fullpage,amscd,amssymb,latexsym,url,hyperref}
\usepackage[all,2cell,ps]{xy}

\theoremstyle{plain}
\newtheorem{thm}{Theorem}[section]
\newtheorem{theorem}[thm]{Theorem}

\newtheorem{corollary}[thm]{Corollary}
\newtheorem{lm}[thm]{Lemma}
\newtheorem{lemma}[thm]{Lemma}
\newtheorem{con}[thm]{Conjecture}
\newtheorem{prop}[thm]{Proposition}
\newtheorem{proposition}[thm]{Proposition}

\theoremstyle{definition}
\newtheorem{de}[thm]{Definition}
\newtheorem{exm}[thm]{Example}
\newtheorem{rem}[thm]{Remark}
\newtheorem{remark}[thm]{Remark}
\newtheorem{example}[thm]{Example}

\newcommand{\Z}{\mathbb{Z}}

\newcommand{\A}{\mathcal{A}}

\newcommand{\im}{\mathrm{Im}}

\newcommand{\Aff}{\mathrm{Aff}}
\newcommand{\aff}[1]{\mathrm{Aff}(#1)}
\newcommand{\aut}[1]{\mathrm{Aut}(#1)}
\newcommand{\dis}[1]{\mathrm{Dis}(#1)}
\newcommand{\lmlt}[1]{\mathrm{LMlt}(#1)}

\newcommand{\ld}{\backslash}
\newcommand{\orb}[2]{\mathrm{Orb}_{#1}(#2)}

\numberwithin{equation}{section}

\begin{document}

\title{The structure of medial quandles}

\author{P\v remysl Jedli\v cka}
\author{Agata Pilitowska}
\author{David Stanovsk\'y}
\author{Anna Zamojska-Dzienio}

\address{(P.J.) Department of Mathematics, Faculty of Engineering, Czech University of Life Sciences, Kam\'yck\'a 129, 16521 Praha 6, Czech Republic}
\address{(A.P., A.Z.) Faculty of Mathematics and Information Science, Warsaw University of Technology, Koszykowa 75, 00-662 Warsaw, Poland}
\address{(D.S.) Department of Algebra, Faculty of Mathematics and Physics, Charles University, Sokolovsk\'a 83, 18675 Praha 8, Czech Republic}

\email{(P.J.) jedlickap@tf.czu.cz} 
\email{(A.P.) apili@mini.pw.edu.pl}
\email{(D.S.) stanovsk@karlin.mff.cuni.cz}
\email{(A.Z.) A.Zamojska-Dzienio@mini.pw.edu.pl}

\thanks{Joint research within the framework of the Czech-Polish cooperation grants 7AMB13PL013 and 8829/R13/R14. The third author was partly supported by the GA\v CR grant 13-01832S}


\date{\today}

\begin{abstract}
Medial quandles are represented using a heterogeneous affine
structure. As a consequence, we obtain numerous structural
properties, including enumeration of isomorphism classes of medial
quandles up to 13 elements.
\end{abstract}

\maketitle

\section{Introduction}\label{sec1}

An algebraic structure $(Q,\cdot)$ is called a \emph{quandle} if the following conditions hold, for every $x,y,z\in Q$:
\begin{itemize}
\item $xx=x$ (we say $Q$ is \emph{idempotent}),
\item $x(yz)=(xy)(xz)$ (we say $Q$ is \emph{left distributive}),
\item the equation $xu=y$ has a unique solution $u\in Q$ (we say $Q$ is a \emph{left quasigroup}).
\end{itemize}
Among the many motivations behind quandles, perhaps the most striking is the one coming from knot theory: the three axioms of quandles correspond to the three Reidemeister moves \cite{J82}. See \cite{AG,HSV} for an introduction to the algebraic theory of quandles, and \cite{Car,Kam} for a knot-theoretical perspective.

A quandle $Q$ is called \emph{medial} if $$(xy)(uv)=(xu)(yv)$$ for every $x,y,u,v\in Q$.
(In some papers, the adjective \emph{abelian} is used; this word is somewhat overloaded in mathematics, and we object to use it for a reason explained below.)
The most important examples are \emph{affine quandles} $\Aff(A,f)$ (also called \emph{Alexander quandles} elsewhere), constructed over any abelian group $A$ with an automorphism $f$ by taking the operation $x*y=(1-f)(x)+f(y)$. A detailed study of the structure of affine quandles has been encountered in \cite{H,Hou1,Hou2}. However, we are not aware of any paper devoted to the structure of medial quandles in general. The main purpose of the present paper is to show the rich structure of medial quandles.

Our motivation is twofold. First, mediality defines an important class of quandles, related to the abstract notion of abelianess. Medial quandles are precisely the abelian quandles in the sense of the Higgins commutator theory \cite{Hig}. In other terms, they are the intersection of the class of quandles and the class of \emph{modes} \cite{RS}. Medial quandles are close to being abelian in the sense of the Smith commutator theory \cite{FM}: the orbit decomposition is an abelian congruence.
(This is why we prefer using the adjective `medial', which in turn dates back to the 1940's.)
We plan to study the abstract commutator theory connections in a subsequent paper.
The second motivation is our belief that our methods can be adapted to general quandles, combining the present approach with the theory developed for connected quandles in \cite{HSV}. A proof of concept can be found in \cite{Pie} for the special case of involutory quandles.

Most medial quandles are not affine, this is the multiplication table of the smallest example:
$$\begin{array}{c|ccc}
 &0&1&2\\\hline
0&0&1&2\\
1&0&1&2\\
2&1&0&2\\
\end{array}$$
Our main result, Theorem \ref{thm:decomposition}, states that all medial quandles are built from affine pieces using a heterogeneous affine structure, called \emph{affine mesh} here. The affine pieces correspond to the orbits under a certain group of automorphisms, the multiplication group. In the example, the orbit $\{0,1\}$ is in fact $\Aff(\Z_2,1)$, and the orbit $\{2\}$ is the trivial affine quandle.

The concept of affine meshes turns out to be a powerful tool. As an application, we obtain several structural results about medial quandles (see, e.g., Theorems \ref{thm:latin_orbits}, \ref{cor:copr3}, \ref{thm:reductive}, \ref{cor:copr}, \ref{cor:2-symmetric*}), reveal a hierarchy with respect to algebraic properties (Section \ref{sec:reductivity}), and, perhaps most interestingly, enumerate isomorphism classes of medial quandles up to size 13 (and more, in many interesting cases), thus extending considerably existing enumerations (see the OEIS series A165200 \cite{OEIS}). We also discuss asymptotic enumerations obtained in \cite{B}.

As far as we know, this is the second attempt on a complete orbit decomposition theorem for (a subclass of) quandles, only after \cite{Pie} on involutory quandles.
The orbit decomposition for general quandles was addressed in several papers, most recently in \cite{EGTY,NW}. However, none of the approaches provides the structure of orbits, nor any control over the way the orbits are assembled, nor any isomorphism result. 

Initially, our work was inspired by a series of papers by Roszkowska \cite{R87,R89,R99a,R99b} on involutory medial quandles (called \emph{SIE-groupoids} there). Some of her results are generalized here.

\subsection*{Contents}
In Sections 2--4, we develop the \emph{representation theory}. First, we introduce two important groups acting on a quandle, the multiplication group and the displacement group. Their orbits of transitivity determine a decomposition to subquandles that can be viewed as ``minimal left ideals" (in the sense of semigroup theory). In Section \ref{sec:decomposition}, we prove that all orbits, as subquandles, are affine (Proposition \ref{prop:affine_orbit}), introduce affine meshes and their sums, and prove that every medial quandle can be represented this way (Theorem \ref{thm:decomposition}). In Section \ref{sec:iso} we prove the Isomorphism Theorem \ref{thm:isomorphism} that determines when two meshes represent isomorphic quandles.

In Section \ref{sec:latin_orbits}, we look at medial quandles whose \emph{orbit subquandles are latin squares}, i.e., in every orbit subquandle $\Aff(A,f)$, the endomorphism $1-f$ is an automorphism. This class can be considered as having ``the richest algebraic structure". The main result, Theorem \ref{thm:latin_orbits}, states that all such quandles are direct products of a latin quandle and a projection quandle.

In Section \ref{sec:reductivity}, we develop the notion of \emph{$m$-reductivity} \cite{PR}, stating that in every orbit subquandle
$\Aff(A,f)$, the endomorphism $1-f$ is nilpotent of degree at most $m-1$ (Theorem \ref{thm:reductive}).
As a consequence, we show some limitations on the orbit sizes of medial quandles that are not $m$-reductive for a small $m$ (Corollaries \ref{cor:copr3} and \ref{cor:copr}).
The extreme case, 2-reductivity, refers to quandles where all orbits are projection quandles $\Aff(A,1)$, hence have ``the poorest algebraic structure". This class was investigated (under the name \emph{cyclic modes}) by P\l onka, Romanowska and Roszkowska in \cite{Plo,RR89}. Our representation theorem, Theorem
\ref{thm:2-reductive}, generalizes the one given in \cite[Section 2]{RR89}.

In Section \ref{sec:symmetry}, we apply the representation theory to
medial quandles with a bound on the order of translations. In
particular, we address the structure of \emph{involutory
quandles} (or \emph{keis}), where all translations have order at
most 2, and obtain the results of \cite{R99a} as a special case.

Section \ref{sec:enumeration} contains results on enumeration of isomorphism classes of medial quandles.
First, in \ref{ssec:asymptotics}, we discuss and somewhat refine Blackburn's results \cite{B} on asymptotic enumeration.
Then, in \ref{ssec:computation}, we present computational results on enumeration of small medial quandles, using algorithms described in \ref{ssec:algorithms}.

In Section \ref{sec:congruences}, we conclude the paper with a note on congruence structure of medial quandles, with an outlook on future work.

\subsection*{Notation and basic terminology}
The identity permutation will always be denoted by 1.
For two permutations $\alpha,\beta$, we write $\alpha^\beta=\beta\alpha\beta^{-1}$.
The commutator is defined $[\alpha,\beta]=\beta^\alpha\beta^{-1}$.

Let a group $G$ act on a set $X$. For $e\in X$, the stabilizer of $e$ will be denoted $G_e$.

Let $Q=(Q,\cdot)$ be an algebraic structure with a single binary operation (shortly, a \emph{binary algebra}, also called a groupoid or a magma). 
The \emph{left translation} by $a\in Q$ is the mapping $L_a:Q\to Q$, $x\mapsto ax$. 
If $Q$ is a left quasigroup, the unique solution to $au=b$ will be denoted by $u=a\ld b$, and we have $L_a^{-1}(x)=a\ld x$. 
Observe that $Q$ is left distributive iff all left translations are endomorphisms, and $Q$ is a left quasigroup iff all left translations are permutations.
We will often use the following observation: for every $a\in Q$ and $\alpha\in\aut Q$,
\begin{equation}
(L_a)^\alpha=L_{\alpha(a)}.\label{auto}
\end{equation}
Occasionally, we will also use \emph{right translations} $R_a(x)=xa$.

A \emph{subquandle} is a subset closed with respect to both operations $\cdot$ and $\ld$. Note that finite subsets closed with respect to $\cdot$ are always subquandles.

\section{The displacement group}\label{sec:dis}

The (left) \emph{multiplication group} of a quandle $Q$ is the permutation group generated by left translations, i.e.,
\begin{equation*}
{\rm LMlt}(Q)=\langle L_a\mid a\in Q\rangle\leq S_Q.
\end{equation*}
We define the \emph{displacement group} as the subgroup
\begin{equation*}
{\rm Dis}(Q)=\langle L_aL_b^{-1}\mid a,b\in Q\rangle.
\end{equation*}
Using the fact that all translations are automorphisms of $Q$, together with equality \eqref{auto}, we obtain that both $\lmlt Q$ and $\dis Q$ are normal subgroups of $\aut Q$. (Various names are used in literature for the groups $\lmlt Q$ and $\dis Q$. E.g., Joyce \cite{J82} uses \emph{inner automorphism group} and \emph{transvection group}, respectively, and translations are called \emph{inner mappings}.)

An important lesson learnt in \cite{HSV} is that many properties of quandles are determined by the properties of their displacement groups.
The following facts will be used extensively throughout the paper without explicit reference (all ideas in Proposition \ref{Pr:DQ} appeared already in \cite{J82,J82b}).

\begin{proposition}\label{Pr:DQ}
Let $Q$ be a quandle. Then
\begin{enumerate}
\item $\dis Q = \{L_{a_1}^{k_1}\dots L_{a_n}^{k_n}:\ a_1,\dots,a_n\in Q$ and $\sum_{i=1}^n k_i=0\}$;
\item the natural actions of $\lmlt{Q}$ and $\dis{Q}$ on $Q$ have the same orbits;
\item $Q$ is medial if and only if $\dis Q$ is abelian.
\end{enumerate}
\end{proposition}

\begin{proof}
%
(1) Let $S$ be the set on the right-hand side of the expression. Since the generators of $\dis Q$ belong to $S$, we have $\dis Q\subseteq S$. For the other inclusion,
we note that every $\alpha\in S$ can be written as $L_{a_1}^{k_1}\dots L_{a_n}^{k_n}$, where not only $\sum_i k_i=0$ but also $k_i=\pm 1$. Assuming such a decomposition, we prove by induction on $n$ that $\alpha\in\dis Q$.
If $n=0$ then $\alpha$ is the identity, the case $n=1$ does not occur, and if $n=2$ we have either $\alpha=L_aL_b^{-1}\in\dis Q$, or $\alpha=L_a^{-1}L_b=L_{a\ld b}L_a^{-1}\in\dis Q$. Suppose that $n>2$.

If $k_1=k_n$ then there is $1<m<n$ such that $\sum_{i<m}k_i=0$ and $\sum_{i\geq m}k_i=0$. Let $\beta=L_{a_1}^{k_1}\ldots L_{a_{m-1}}^{k_{m-1}}$ and $\gamma=L_{a_m}^{k_m}\ldots L_{a_n}^{k_n}$. Then, by the induction assumption, $\beta,\gamma\in\dis Q$, and so $\alpha=\beta\gamma\in\dis Q$.

If $k_1\neq k_n$ then $$\alpha=L_a^k\beta L_b^{-k}=L_a^k(\beta L_b^{-k}\beta^{-1})\beta=(L_a^kL_{\beta(b)}^{-k})\beta$$ for some $a,b\in Q$, $k\in\{\pm1\}$ and $\beta=L_{a_2}^{k_2}\ldots L_{a_{n-1}}^{k_{n-1}}$ such that $\sum_{2\leq i\leq n-1}k_i=0$. Since both $L_a^kL_{\beta(b)}^{-k}$ and $\beta$ belong to $\dis Q$, we get $\alpha\in\dis Q$.

(2) Let $x,y$ be two elements in a single orbit of $\lmlt Q$ such that $y=\alpha(x)$ with $\alpha=L_{a_1}^{k_1}\dots L_{a_n}^{k_n}\in\lmlt Q$. With $k=k_1+\cdots+k_n$, we have $\beta = L_y^{-k}\alpha\in\dis Q$ by (1), and $\beta(x) = L_y^{-k}\alpha(x)= L_y^{-k}(y) = y$.

(3) $Q$ is medial iff $L_{xy}L_z=L_{xz}L_y$ for every $x,y,z\in Q$, and by expanding $L_{xy}=L_xL_yL_x^{-1}$, and similarly for $L_{xz}$, we obtain that $Q$ is medial iff
\begin{equation}\label{Eq:DQ2}
L_yL_x^{-1}L_z=L_zL_x^{-1}L_y
\end{equation}
for every $x,y,z\in Q$.
$(\Leftarrow)$ If $\dis Q$ is abelian then $L_yL_x^{-1}L_zL_y^{-1}=L_zL_y^{-1}L_yL_x^{-1}=L_zL_x^{-1}$ for every $x,y,z\in Q$, and we obtain \eqref{Eq:DQ2}.
$(\Rightarrow)$ Conversely, starting with \eqref{Eq:DQ2}, we obtain $L_y^{-1}L_xL_z^{-1}=L_z^{-1}L_xL_y^{-1}$ for every $x,y,z\in Q$, and thus $L_xL_y^{-1}L_uL_v^{-1}=L_uL_y^{-1}L_xL_v^{-1}=L_uL_v^{-1}L_xL_y^{-1}$ for every $x,y,u,v\in Q$, proving that $\dis Q$ is abelian.
\end{proof}

We will refer to the orbits of transitivity of the groups $\lmlt Q$ and $\dis Q$ simply as \emph{the orbits of $Q$}, and denote $$Qe=\{\alpha(e)\mid\alpha\in\lmlt Q\}=\{\alpha(e)\mid\alpha\in\dis Q\}$$ the orbit containing an element $e\in Q$.
Notice that orbits are subquandles of $Q$: for $\alpha(e),\beta(e)\in Qe$ with $\alpha,\beta\in\lmlt Q$, we have $\alpha(e)\cdot\beta(e)=(L_{\alpha(e)}\beta)(e)\in Qe$ and $\alpha(e)\ld\beta(e)=(L_{\alpha(e)}^{-1}\beta)(e)\in Qe$.

A quandle is called \emph{connected}, if it consists of a single orbit. Orbits (as subquandles) are not necessarily connected.
A quandle is called \emph{latin} (or, a \emph{quasigroup}), if the right translations, $R_a:Q\to Q$, $x\mapsto xa$, are bijective, too. Latin quandles are obviously connected.
Connected quandles were studied in detail in \cite{HSV}. In particular, it was proved there that connected medial quandles are affine, see also
Corollary \ref{cor_medial iff affine}.

\begin{example}
Let $A$ be an abelian group, $f$ its endomorphism, and define an operation on the set $A$ by
\begin{equation*}
a*b=(1-f)(a)+f(b).
\end{equation*}
The resulting binary algebra $\aff{A,f}=(A,*)$ is called \emph{affine} over the group $A$, and is easily shown to be idempotent and medial. If $f$ is an automorphism then it is a medial quandle, called \emph{affine quandle} over $A$. Notice the equation
\begin{equation*}
a\setminus b=L_a^{-1}(b)=(1-f^{-1})(a)+f^{-1}(b).
\end{equation*}
Any non-empty set with operation $a\cdot b=b$ is a medial quandle, called \emph{right projection quandle}. It is affine with $f=1$.
\end{example}

An alternative definition of affine quandles can be given in terms of modules: every affine quandle results from a module over the ring $\Z[t,t^{-1}]$ of Laurent series over the integers. The relation between affine quandles and $\Z[t,t^{-1}]$-modules is explained in detail in \cite{H,Hou2}.

It is not difficult to calculate that $$\dis{\Aff(A,f)}=\{x\mapsto x+a: a\in\im(1-f)\}\simeq\im(1-f),$$ hence $\Aff(A,f)$ is connected iff $1-f$ is onto. Clearly, $\Aff(A,f)$ is latin iff $1-f$ is a permutation, hence finite connected affine quandles are always latin.

\begin{rem}
Our main result, Theorem \ref{thm:decomposition}, shows that for every medial quandle, there is a congruence (namely, the orbit decomposition) such that all blocks are affine quandles and the factor is a right projection quandle. A complementary approach is suggested in \cite[Theorem 8.6.13]{RS}: for a medial quandle $Q$ and a fixed element $e\in Q$, consider the mapping
$$\varphi:Q\to\aff{\dis Q,\psi_e},\quad a\mapsto L_aL_e^{-1},$$
where $\psi_e(\alpha)=\alpha^{L_e}$. It is not difficult to check that $\varphi$ is an onto homomorphism, hence $Q/\mathrm{ker}(\varphi)$ is an affine quandle and the blocks of the kernel are right projection quandles.
\end{rem}

\section{Orbit decomposition}\label{sec:decomposition}

Let $Q$ be a medial quandle and $e\in Q$. There is a bijection between the elements of the orbit $Qe=\{\alpha(e)\mid\alpha\in\dis Q\}$, and the elements of the abelian group $\dis Q/\dis Q_e$, with the coset $\alpha\dis Q_e$ corresponding to the element $\alpha(e)$. This justifies the following definition (which makes sense in a much wider setting and can be traced back to \cite[Corollary 2.7]{Kep}).

\begin{de}
Let $\alpha(e),\beta(e)\in Qe$ with $\alpha,\beta\in\dis Q$ and put
$$\alpha(e)+\beta(e)=\alpha\beta(e)\qquad\text{and}\qquad -\alpha(e)=\alpha^{-1}(e).$$
Then $\orb Qe=(Qe,+,-,e)$ is an abelian group, called the \emph{orbit group} for $Qe$.
\end{de}

Clearly, if $Qe=Qf$, we have $\orb Qe\simeq\dis Q/\dis Q_e\simeq\dis Q/\dis Q_f\simeq\orb Qf$. In fact, as we shall see, every $\lambda\in\lmlt Q$ acts as an isomorphism.

\begin{lemma}\label{lem:autQe}
Let $Q$ be a medial quandle, $e\in Q$ and $\lambda\in\lmlt Q$. Then $\lambda$ is an isomorphism $\orb Qe\simeq\orb Q{\lambda(e)}$.
\end{lemma}

\begin{proof}
Let $\alpha(e), \beta(e)\in Qe$ with $\alpha,\beta\in\dis Q$. First notice that $$\lambda(\alpha(e))=\lambda\alpha
\lambda^{-1}\lambda(e)=\alpha^{\lambda}(\lambda(e)).$$
It follows immediately that $\lambda$ maps $Qe$ into $Q\lambda(e)$. The mapping $\lambda$ is injective, and for every $\gamma\in\dis Q$ we have
$\gamma(\lambda(e))=\lambda\gamma^{\lambda^{-1}}(e)\in\lambda(Qe)$, hence it is a bijection between $Qe$ and $Q\lambda(e)$.

It remains to show that $\lambda$ preserves the addition.
On one side, we have
$\lambda(\alpha(e)+\beta(e))=\lambda(\alpha\beta(e))=(\alpha\beta)^{\lambda}(\lambda(e))$.
On the other side, we have
$\lambda(\alpha(e))+\lambda(\beta(e))=\alpha^{\lambda}(\lambda(e))+\beta^{\lambda}(\lambda(e))=\alpha^{\lambda}\beta^{\lambda}(\lambda(e))$,
and we see the two sides are equal.
\end{proof}

It follows that the translation $L_e$ is an automorphism of the group $\orb Qe$, for every $e\in Q$. We are ready to prove the first important step towards the decomposition theorem: every orbit of a medial quandle is an affine quandle.

\begin{prop}\label{affine}\label{prop:affine_orbit}
Let $Q$ be a medial quandle and $e\in Q$. Then $Qe=\aff{\orb Qe,L_e}$.
\end{prop}

\begin{proof}
Let $a,b\in Qe$, write $a=\alpha(e)$, $b=\beta(e)$ for some $\alpha,\beta\in\dis Q$. We want to prove that $$a\cdot b=(1-L_e)(a)+L_e(b).$$
Write $(1-L_e)(a)+L_e(b)=\alpha(e)-L_e\alpha(e)+L_e\beta(e)=\alpha(e)-\alpha^{L_e}(e)+\beta^{L_e}(e)$, and using the fact that both $\alpha^{L_e},\beta^{L_e}\in\dis Q$, we can rewrite the right-hand side as $\alpha(\alpha^{L_e})^{-1}\beta^{L_e}(e)=\alpha L_e\alpha^{-1}(L_e)^{-1}L_e\beta L_e^{-1}(e)=L_{\alpha(e)}\beta(e)=L_a(b)=a\cdot b$.
\end{proof}

\begin{corollary}\cite[Section 5]{HSV}\label{cor_medial iff affine}
A connected quandle is medial if and only if it is affine.
\end{corollary}

\begin{example}\label{ex:6elt}
Let $Q=\aff{\Z_6,-1}$. The multiplication table can be written as follows:
$$\begin{array}{c|ccc|ccc}
 &0&2&4&1&3&5\\\hline
0&0&4&2&5&3&1\\
2&4&2&0&3&1&5\\
4&2&0&4&1&5&3\\
1&2&0&4&1&5&3\\
3&0&4&2&5&3&1\\
5&4&2&0&3&1&5\\
\end{array}$$
We immediately see that there are two orbits, $Q0$ and $Q1$. Calculate
\begin{align*}
\lmlt Q&=\langle(2\ 4)(1\ 5),(0\ 4)(1\ 3),(0\ 2)(3\ 5)\rangle,\\
\dis Q&=\langle(0\ 4\ 2)(1\ 5\ 3)\rangle,
\end{align*}
and observe that $\dis Q_0=\dis Q_1=\{1\}$. Hence $\orb Q0\simeq\dis Q/\dis Q_0\simeq\Z_3$, where $L_0$ acts on the group
$\Z_3$ as $-1$, and analogously for $Q1$. We obtain $Q0\simeq Q1\simeq\aff{\Z_3,-1}$.
\end{example}

The group structure of the orbits motivates the following two important definitions.

\begin{de}
An \emph{affine mesh} over a non-empty set $I$ is a triple
$$\mathcal A=((A_i)_{i\in I},\,(\varphi_{i,j})_{i,j\in I},\,(c_{i,j})_{i,j\in I})$$ where $A_i$ are abelian groups, $\varphi_{i,j}:A_i\to A_j$ homomorphisms, and $c_{i,j}\in A_j$ constants, satisfying the following conditions for every $i,j,j',k\in I$:
\begin{enumerate}
    \item[(M1)] $1-\varphi_{i,i}$ is an automorphism of $A_i$;
    \item[(M2)] $c_{i,i}=0$;
    \item[(M3)] $\varphi_{j,k}\varphi_{i,j}=\varphi_{j',k}\varphi_{i,j'}$, i.e., the following diagram commutes:
$$\begin{CD}
A_i @>\varphi_{i,j}>> A_j\\ @VV\varphi_{i,j'}V @VV\varphi_{j,k}V\\
A_{j'} @>\varphi_{j',k}>> A_k
\end{CD}$$
    \item[(M4)] $\varphi_{j,k}(c_{i,j})=\varphi_{k,k}(c_{i,k}-c_{j,k})$.
\end{enumerate}
If the index set is clear from the context, we shall write briefly $\mathcal A=(A_i;\varphi_{i,j};c_{i,j})$.
\end{de}

\begin{de}
The \emph{sum of an affine mesh} $(A_i;\varphi_{i,j};c_{i,j})$ over a set $I$ is a binary algebra defined on the disjoint union of the sets $A_i$, with operation
$$a*b=c_{i,j}+\varphi_{i,j}(a)+(1-\varphi_{j,j})(b)$$
for every $a\in A_i$ and $b\in A_j$.
\end{de}

Notice that every fibre $A_i$ becomes a subquandle of the sum, and for $a,b\in A_i$ we have
$$a*b=\varphi_{i,i}(a)+(1-\varphi_{i,i})(b),$$
hence $(A_i,*)$ is affine and equal to $\aff{A_i,1-\varphi_{i,i}}$.

\begin{lemma}\label{lem:mesh1}
The sum of an affine mesh is a medial quandle.
\end{lemma}

\begin{proof}
For idempotence, $a*a=c_{i,i}+\varphi_{i,i}(a)+(1-\varphi_{i,i})(a)=a$ for every $a\in A_i$.
For the left quasigroup property, notice that the equation $a*x=c_{i,j}+\varphi_{i,j}(a)+(1-\varphi_{j,j})(x)=b$ with $a\in A_i$, $b\in A_j$, has a unique solution in $A$, namely $$x=(1-\varphi_{j,j})^{-1}(b-\varphi_{i,j}(a)-c_{i,j})\in A_j.$$
For mediality, with $a\in A_i$, $b\in A_j$, $c\in A_k$, $d\in A_l$, calculate
\begin{gather*}
(a*b)*(c*d)=\varphi_{j,l}(c_{i,j})+(1-\varphi_{l,l})(c_{k,l})+c_{j,l}+\qquad\qquad\qquad\qquad\\
                        \qquad\qquad\qquad\qquad\varphi_{j,l}(\varphi_{i,j}(a)+(1-\varphi_{j,j})(b)) + (1-\varphi_{l,l})(\varphi_{k,l}(c)+(1-\varphi_{l,l})(d)),
\end{gather*}
and
\begin{gather*}
(a*c)*(b*d)=\varphi_{k,l}(c_{i,k})+(1-\varphi_{l,l})(c_{j,l})+c_{k,l}+\qquad\qquad\qquad\qquad\\
                        \qquad\qquad\qquad\qquad\varphi_{k,l}(\varphi_{i,k}(a)+(1-\varphi_{k,k})(c)) + (1-\varphi_{l,l})(\varphi_{j,l}(b)+(1-\varphi_{l,l})(d)).
\end{gather*}
The equality easily follows from (M3) and (M4). Left distributivity is an obvious consequence of mediality and idempotence.
\end{proof}

We will prove later that every medial quandle is the sum of an affine mesh. Nevertheless, the representation has a problem: the orbits of the sum need not coincide with the sets $A_i$, $i\in I$. For instance, taking $\varphi_{i,j}=0$ and $c_{i,j}=0$ for every $i,j$, we obtain the right projection quandle, where every singleton is an orbit. We need a notion of indecomposability of a mesh.

\begin{de}
An affine mesh $(A_i;\varphi_{i,j};c_{i,j})$ over a set $I$ is called \emph{indecomposable} if
$$A_j=\left<\bigcup_{i\in I}\left(c_{i,j}+\im(\varphi_{i,j})\right)\right>,$$ for every $j\in I$.
Equivalently, the group $A_j$ is generated by all elements $c_{i,j}$, $\varphi_{i,j}(a)$ with $i\in I$ and $a\in A_i$.
\end{de}

\begin{lemma}\label{lem:mesh2}
The sum of an indecomposable affine mesh $(A_i;\varphi_{i,j};c_{i,j})$ over a set $I$ is a medial quandle with orbits $A_i$, $i\in I$.
\end{lemma}

\begin{proof}
We calculate the restriction $\dis Q|_{A_j}$ of the group $\dis Q$ on the subset $A_j$. For $x\in A_j$, $a\in A_k$ and $b\in A_l$ we have
\begin{align*}
L_a(x)&=c_{k,j}+\varphi_{k,j}(a)+(1-\varphi_{j,j})(x)\\
L_b^{-1}(x)&=(1-\varphi_{j,j})^{-1}(x-c_{l,j}-\varphi_{l,j}(b))
\end{align*}
and thus
$$L_aL_b^{-1}(x)=c_{k,j}-c_{l,j}+\varphi_{k,j}(a)-\varphi_{l,j}(b)+x.$$
Taking $k=i$, $l=j$, $a=0$ in $A_i$ and $b=0$ in $A_j$, we obtain the mapping $\alpha_i(x)=c_{i,j}+x$.
Taking $k=l=i$, $a\in A_i$ and $b=0$ in $A_i$, we obtain the mapping $\beta_{i,a}(x)=\varphi_{i,j}(a)+x$.
We see that the mappings $\alpha_i$ and $\beta_{i,a}$ generate the group $\dis Q|_{A_j}$.

Now notice that $\dis Q|_{A_j}$ is in fact a subgroup of the group $A_j$ acting on itself by translations. Hence it is transitive on $A_j$ if and only if it equals to $A_j$. This happens if and only if the elements $c_{i,j}$ (acting as mappings $\alpha_i$) and $\varphi_{i,j}(a)$ (acting as mappings $\beta_{i,a}$) generate the group $A_j$.
\end{proof}

\begin{example}
Consider the quandle $Q$ from Example \ref{ex:6elt}. We can represent it as the sum of an affine mesh in two ways:
\begin{itemize}
    \item Using the representation $Q=\aff{\Z_6,-1}$, we see $Q$ is the sum of the mesh $((\Z_6),\,(2),\,(0))$. However, this mesh is not indecomposable, $Q$ has two orbits.
    \item Using the orbit representation, we see $Q$ is the sum of the mesh
    $((\Z_3,\Z_3),\,\left(\begin{smallmatrix}2&2\\2&2\end{smallmatrix}\right),\,\left(\begin{smallmatrix}0&2\\1&0\end{smallmatrix}\right))$.
    This mesh is indecomposable.
\end{itemize}
\end{example}

The latter representation motivates the following definition.

\begin{de}
Let $Q$ be a medial quandle, and choose a transversal $E$ to the orbit decomposition. We define the \emph{canonical mesh} for $Q$ over the transversal $E$ as $\A_{Q,E}=(\orb Qe;\varphi_{e,f};c_{e,f})$ with $e,f\in E$, where for every $x\in Qe$
$$\varphi_{e,f}(x)=xf-ef\quad\text{and}\quad c_{e,f}=ef.$$
\end{de}

We will soon prove that $\A_{Q,E}$ is really an affine mesh. While it depends on the transversal $E$, all canonical meshes for $Q$ are ``similar", in a sense to be specified in the next section. Therefore, we will often say ``a canonical mesh of $Q$", but we really mean ``the canonical mesh for $Q$ over a transversal~$E$".

To simplify calculations below, we will use the following observation: for every $\alpha\in\dis Q$,
$$\varphi_{e,f}(\alpha(e))=[\alpha,L_e](f).$$
Indeed, $\varphi_{e,f}(\alpha(e))=\alpha(e)f-ef=L_{\alpha(e)}L_f^{-1}(f)-L_eL_f^{-1}(f)=L_{\alpha(e)}L_f^{-1}L_fL_e^{-1}(f)$, and using \eqref{auto}, we obtain $\alpha L_e\alpha^{-1}L_e^{-1}(f)=[\alpha,L_e](f)$.

\begin{lemma}\label{lem:mesh3}
Let $Q$ be a medial quandle and $\A_{Q,E}$ its canonical mesh. Then $\A_{Q,E}$ is an indecomposable affine mesh and $Q$ is equal to its sum.
\end{lemma}

\begin{proof}
First notice that the orbit groups $\orb Qe$ are abelian groups with an underlying set $Qe$ (Proposition \ref{prop:affine_orbit}), the constants $c_{e,f}$ are in $Qf$, and we verify that the mappings $\varphi_{e,f}$ are homomorphisms $\orb Qe\to\orb Qf$. For $\alpha(e),\beta(e)\in Qe$ with $\alpha,\beta\in\dis Q$, we have
\begin{align*}
\varphi_{e,f}(\alpha(e))+\varphi_{e,f}(\beta(e))&=[\alpha,L_e](f)+[\beta,L_e](f)=[\alpha,L_e][\beta,L_e](f),\\
\varphi_{e,f}(\alpha(e)+\beta(e))&=\varphi_{e,f}(\alpha\beta(e))=[\alpha\beta,L_e](f),
\end{align*}
and using commutativity of $\dis Q$, we see that $$[\alpha,L_e][\beta,L_e]=\alpha(\alpha^{-1})^{L_e}\beta(\beta^{-1})^{L_e}=\alpha\beta(\alpha^{-1})^{L_e}(\beta^{-1})^{L_e}=[\alpha\beta,L_e].$$

Now we verify the properties (M1) to (M4). For (M1),
$$(1-\varphi_{e,e})(\alpha(e))=\alpha(e)-[\alpha,L_e](e)=\alpha[\alpha,L_e]^{-1}(e)=L_e(\alpha(e)),$$
hence $1-\varphi_{e,e}=L_e\in\aut{\orb Qe}$ according to Lemma \ref{lem:autQe}. In the last step, we again used commutativity of $\dis Q$ to show that
$$\alpha[\alpha,L_e]^{-1}=\alpha\alpha^{L_e}\alpha^{-1}=\alpha\alpha^{-1}\alpha^{L_e}=\alpha^{L_e}.$$
For (M2), we only notice that $c_{e,e}=e$ which is the zero element in $\orb Qe$. For (M3),
$$\varphi_{f,g}\varphi_{e,f}(\alpha(e))=\varphi_{f,g}([\alpha,L_e](f))=[[\alpha,L_e],L_f](g)=L_e^\alpha[\alpha,L_e]^{-1}L_e^{-1}(g),$$
hence is independent of $f$. Again, in the last step, commutativity yields
$$[[\alpha,L_e],L_f]=L_e^\alpha(L_e^{-1}L_f)[\alpha,L_e]^{-1}L_f^{-1}=L_e^\alpha[\alpha,L_e]^{-1}(L_e^{-1}L_f)L_f^{-1}=L_e^\alpha[\alpha,L_e]^{-1}L_e^{-1}.$$
For (M4),
\begin{align*}
\varphi_{f,g}(c_{e,f})&=\varphi_{f,g}(L_eL_f^{-1}(f))=[L_eL_f^{-1},L_f](g)=[L_e,L_f](g),\\
\varphi_{g,g}(c_{e,g}-c_{f,g})&=\varphi_{g,g}(L_eL_f^{-1}(g))=[L_eL_f^{-1},L_g](g),
\end{align*}
and, using commutativity again,
$$[L_eL_f^{-1},L_g]=L_eL_f^{-1}L_g(L_eL_f^{-1})^{-1}L_g^{-1}=L_e(L_eL_f^{-1})^{-1}L_f^{-1}L_gL_g^{-1}=[L_e,L_f].$$

Next we show that $\A_{Q,E}$ is indecomposable. Since $\im(\varphi_{e,f})=\{xf-ef:x\in Qe\}$, and $c_{e,f}=ef$, we see that
$c_{e,f}+\im(\varphi_{e,f})=\{xf:x\in Qe\}$, and taking the union we obtain $\bigcup_{e\in E}\{xf:x\in Qe\}=\{xf:x\in Q\}$. This set generates the group $\orb Qf$.

Finally, we verify that the sum yields back the original quandle $Q$: for $a\in Qe$, $b\in Qf$,
$$a*b=c_{e,f}+\varphi_{e,f}(a)+(1-\varphi_{f,f})(b)=ef+af-ef+b-bf+ff=af+b-bf+f,$$
and taking $\beta\in\dis Q$ such that $b=\beta(f)$, we obtain
$$a*b=(L_aL_f^{-1})\beta(L_bL_f^{-1})^{-1}(f)=(L_aL_f^{-1})(L_bL_f^{-1})^{-1}\beta(f)=L_aL_b^{-1}(b)=a\cdot b.$$
\end{proof}

Alternatively, we could have defined the canonical mesh using the groups $A_e=\dis Q/\dis Q_e$, homomorphisms $\varphi_{e,f}(\alpha\dis Q_e)=[\alpha,L_e]\dis Q_f$, and constants $c_{e,f}=L_eL_f^{-1}\dis Q_f$. Then the original quandle $Q$ is \emph{isomorphic} to the sum of the mesh, where the coset $\alpha\dis Q_e$ corresponds to the element $\alpha(e)\in Q$.

\begin{theorem}\label{thm:decomposition}
A binary algebra is a medial quandle if and only if it is the sum of an indecomposable affine mesh. The orbits of the quandle coincide with the groups of the mesh.
\end{theorem}

\begin{proof}
Combine Lemmas \ref{lem:mesh1}, \ref{lem:mesh2} and \ref{lem:mesh3}.
\end{proof}

\begin{example}\label{ex:4}
There are exactly six medial quandles of size 4, up to isomorphism. They are the sums of the following indecomposable affine meshes:
\begin{itemize}
    \item One orbit: $((\Z_2^2),\,(\left(\begin{smallmatrix}1&1\\1&0\end{smallmatrix}\right)),\,(0)).$ (The endomorphism of the only fibre $\Z_2^2$ is given by a matrix.)
        \item Two orbits: $((\Z_3,\Z_1),\,\left(\begin{smallmatrix}0&0\\0&0\end{smallmatrix}\right),\,\left(\begin{smallmatrix}0&0\\1&0\end{smallmatrix}\right))$ and  $((\Z_2,\Z_2),\,\left(\begin{smallmatrix}0&0\\0&0\end{smallmatrix}\right),\,\left(\begin{smallmatrix}0&1\\1&0\end{smallmatrix}\right))$.
        \item Three orbits: $((\Z_2,\Z_1,\Z_1),\,\left(\begin{smallmatrix}0&0&0\\0&0&0\\0&0&0\end{smallmatrix}\right),\,\left(\begin{smallmatrix}0&0&0\\1&0&0\\c&0&0\end{smallmatrix}\right))$, where $c=0$ or $c=1$.
        \item Four orbits: $((\Z_1,\Z_1,\Z_1,\Z_1),\,0,\,0)$, where $0$ denotes the zero matrix.
\end{itemize}
By a careful analysis using Theorem \ref{thm:isomorphism} (see also Example \ref{ex:4iso}), one can prove that this is a complete list, and that the quandles are pairwise non-isomorphic.
\end{example}

We conclude the section with an easy fact that helps to cut the search space in the enumeration algorithm described in Section \ref{ssec:algorithms}, and will be used also in Section \ref{sec:latin_orbits} to discuss the size of latin orbits.

\begin{prop}\label{prop:gcd}
Let $\mathcal A=(A_i;\,\varphi_{i,j};\,c_{i,j})$ be an affine mesh over a set $I$. Then $|\im(\varphi_{i,i}^2)|$ divides $\gcd(|A_j|:j\in I)$ for every $i\in I$.
\end{prop}

\begin{proof}
Fix $i\in I$. Condition (M3) implies that $\varphi_{i,i}^2=\varphi_{j,i}\varphi_{i,j}$ for every $j\in I$, hence
$$\im(\varphi_{i,i}^2)\leq\im(\varphi_{j,i})\simeq A_j/\mathrm{Ker}(\varphi_{j,i}).$$
Consequently, $|\im(\varphi_{i,i}^2)|$ divides $|A_j|$ for every $j\in I$, hence also their gcd.
\end{proof}

\section{Isomorphism theorem} \label{sec:iso}

\begin{de}
We call two affine meshes $\mathcal A=(A_i;\varphi_{i,j};c_{i,j})$
and $\mathcal A'=(A_i';\varphi_{i,j}';c_{i,j}')$, over the same
index set $I$, \emph{homologous}, if there is a permutation $\pi$ of
the set $I$, group isomorphisms $\psi_i:A_i\to A_{\pi i}'$, and
constants $d_i\in A_{\pi i}'$, such that, for every $i,j\in I$,
\begin{enumerate}
    \item[(H1)] $\psi_j\varphi_{i,j}=\varphi_{\pi i,\pi j}'\psi_i$, i.e., the following diagram commutes:
$$ \begin{CD}
A_i @>\varphi_{i,j}>> A_j\\ @VV\psi_iV @VV\psi_jV\\
A_{\pi i}' @>\varphi_{\pi i,\pi j}'>> A_{\pi j}'
\end{CD}$$
    \item[(H2)] $\psi_j(c_{i,j})=c_{\pi i,\pi j}'+\varphi_{\pi i,\pi j}'(d_i)-\varphi_{\pi j,\pi j}'(d_j)$.
\end{enumerate}
\end{de}

\begin{theorem}\label{thm:isomorphism}
Let $\mathcal A=(A_i;\varphi_{i,j};c_{i,j})$ and $\mathcal A'=(A_i';\varphi_{i,j}';c_{i,j}')$ be two indecomposable affine meshes, over the same index set $I$.
Then the sums of $\A$ and $\A'$ are isomorphic quandles if and only if the meshes $\A$, $\A'$ are homologous.
\end{theorem}

Notice the ``if" implication holds for arbitrary meshes (not just indecomposable).

\begin{proof}
$(\Leftarrow)$
We define a mapping $\psi:\bigcup A_i\to\bigcup A_i'$ by $$\psi(a)=\psi_i(a)+d_i$$ for every $a\in A_i$, and prove that $\psi$ is a quandle isomorphism between the sums. It is clearly a bijection. Let $a\in A_i$, $b\in A_j$. On one side, using the fact that $\psi_j$ is a group homomorphism,
\begin{align*}
\psi(a*b)=\psi_j(a*b)+d_j
&=\psi_j(c_{i,j}+\varphi_{i,j}(a)+(1-\varphi_{j,j})(b))+d_j\\
&=\left(\psi_j\varphi_{i,j}(a)+\psi_j(1-\varphi_{j,j})(b)\right) + \left(\psi_j(c_{i,j})+d_j\right).
\end{align*}
On the other side,
\begin{align*}
\psi(a)*'\psi(b)&=(\psi_i(a)+d_i)*'(\psi_j(b)+d_j)\\
&=c_{\pi i,\pi j}'+\varphi_{\pi i,\pi j}'(\psi_i(a)+d_i)+(1-\varphi_{\pi j,\pi j}')(\psi_j(b)+d_j)\\
&=\left(\varphi_{\pi i,\pi j}'\psi_i(a)+(1-\varphi_{\pi j,\pi j}')\psi_j(b)\right) + \left(c_{\pi i,\pi j}'+\varphi_{\pi i,\pi j}'(d_i)+(1-\varphi_{\pi j,\pi j}')(d_j)\right).
\end{align*}
We see the two expressions are equal using (H1) in the former summand and (H2) in the latter.

$(\Rightarrow)$
Let $f$ be a quandle isomorphism between the two sums. Since isomorphisms preserve orbits, there is a permutation $\pi$ of $I$ such that $f(A_i)=A_{\pi i}'$ for every $i\in I$. Let $0_i$ denote the zero element in the group $A_i$. Put $d_i=f(0_i)$ and define the mappings
\begin{equation*}
\psi_i\colon A_i\to A_{\pi i}',\qquad x\mapsto f(x)-d_i.
\end{equation*}
First, we derive two auxilliary identities, the latter being a stronger version of (H2). Then, we show that all mappings $\psi_i$ are group isomorphisms and verify condition (H1).

Let $i,j\in I$, $a\in A_i$, $b\in A_j$. Consider the value $f(0_j*b)$.
On one hand, using the definition of $\psi_j$,
$$f(0_j*b)=f((1-\varphi_{j,j})(b))=\psi_j((1-\varphi_{j,j})(b))+d_j.$$
On the other hand, using that $f$ preserves $*$,
\begin{align*}
f(0_j\ast b)=f(0_j)\ast' f(b)=d_j\ast' f(b)&= \varphi_{\pi j,\pi j}'(d_j)+(1-\varphi_{\pi j,\pi j}')(f(b))\\
&=\varphi_{\pi j,\pi j}'(d_j)+(1-\varphi_{\pi j,\pi j}')(\psi_j(b)+d_j)=(1-\varphi_{\pi j,\pi j}')(\psi_j(b))+d_j.
\end{align*}
Cancelling $d_j$, we obtain
\begin{equation}\label{eq:1}
\psi_j((1-\varphi_{j,j})(b))=(1-\varphi_{\pi j,\pi j}')(\psi_j(b)).
\end{equation}

For the next identity, consider the value $f(a*0_j)$. On one hand,
$$f(a*0_j)=f(c_{i,j}+\varphi_{i,j}(a))=\psi_j(c_{i,j}+\varphi_{i,j}(a))+d_j.$$
On the other hand,
\begin{align*}
f(a\ast 0_j)=f(a)\ast' f(0_j)=f(a)\ast' d_j &=c_{\pi i,\pi j}'+ \varphi_{\pi i,\pi j}'(f(a))+(1-\varphi_{\pi j,\pi j}')(d_j)\\
&=c_{\pi i,\pi j}'+ \varphi_{\pi i,\pi j}'(\psi_i(a)+d_i)+(1-\varphi_{\pi j,\pi j}')(d_j).
\end{align*}
Cancelling $d_j$, we obtain
\begin{equation}\label{eq:2}
\psi_j(c_{i,j}+\varphi_{i,j}(a))=c_{\pi i,\pi j}'+ \varphi_{\pi i,\pi j}'(\psi_i(a)+d_i)-\varphi_{\pi j,\pi j}'(d_j).
\end{equation}
Setting $a=0_i$, we immediately obtain condition (H2).

To verify that the mappings $\psi_j$ are automorphisms, consider a general product $f(a*b)$. On one hand,
$$f(a*b)=f(c_{i,j}+ \varphi_{i,j}(a)+(1-\varphi_{j,j})(b))=\psi_j(c_{i,j}+ \varphi_{i,j}(a)+(1-\varphi_{j,j})(b))+d_j.$$
On the other hand,
\begin{align*}
f(a\ast b)=f(a)\ast' f(b)
&=c_{\pi i,\pi j}'+ \varphi_{\pi i,\pi j}'(\psi_i(a)+d_i)+(1-\varphi_{\pi j,\pi j}')(\psi_j(b)+d_j)\\
&\stackrel{\eqref{eq:2}}{=} \psi_j(c_{i,j}+\varphi_{i,j}(a))+(1-\varphi_{\pi j,\pi j}')(\psi_j(b))+d_j\\
&\stackrel{\eqref{eq:1}}{=} \psi_j(c_{i,j}+\varphi_{i,j}(a))+\psi_j((1-\varphi_{j,j})(b))+d_j.
\end{align*}
Cancelling $d_j$, substituting $y=(1-\varphi_{j,j})(b)$, and using the fact that $1-\varphi_{j,j}$ is a permutation, we obtain
\begin{equation}\label{eq:3}
\psi_j(c_{i,j}+\varphi_{i,j}(a)+y)=
\psi_j(c_{i,j}+\varphi_{i,j}(a))+\psi_j(y)
\end{equation}
for every $a\in A_i$ and every $y\in A_j$.
Assuming the mesh is indecomposable, every group $A_j$ is generated by all elements $c_{i,j}+\varphi_{i,j}(a)$, $i\in I$, $a\in A_i$.
Hence \eqref{eq:3} implies $\psi_j(x+y)=\psi_j(x)+\psi_j(y)$ for every $x,y\in A_j$, i.e., $\psi_j$ is an automorphism.

Now, we can reuse equation \eqref{eq:2}: expand both sides using the fact that both $\psi_j$ and $\varphi_{\pi i,\pi j}'$ are homomorphisms, obtaining
$$\psi_j(c_{i,j})+\psi_j(\varphi_{i,j}(a))=c_{\pi i,\pi j}'+ \varphi_{\pi i,\pi j}'(\psi_i(a))+\varphi_{\pi i,\pi j}'(d_i)-\varphi_{\pi j,\pi j}'(d_j),$$
and use (H2) to cancel, obtaining $\psi_j(\varphi_{i,j}(a))=\varphi_{\pi i,\pi j}'(\psi_i(a))$, i.e., condition (H1).
\end{proof}

\begin{corollary}\label{c:iso_conn}
Two connected affine quandles $\Aff(A,f)$, $\Aff(B,g)$ are isomorphic if and only if there is a group isomorphism $\psi:A\to B$ such that $g=f^\psi$.
\end{corollary}

\begin{proof}
The statement refers to the case $I=\{1\}$, $\varphi_{1,1}=1-f$, $\varphi_{1,1}'=1-g$. Condition (H1) is equivalent to $g=f^\psi$. Condition (H2) is satisfied trivially regardless the value of $d_1$, because $c_{1,1}=0$ and $c_{1,1}'=0$.
\end{proof}

\begin{example}\label{ex:4iso}
We illustrate the theorem on some of the quandles of size 4, see also Example \ref{ex:4}.
\begin{itemize}
    \item Consider two meshes
$$((\Z_2^2),\,(\left(\begin{smallmatrix}1&1\\1&0\end{smallmatrix}\right)),\,(0))
\quad\text{ and }\quad
((\Z_2^2),\,(\left(\begin{smallmatrix}0&1\\1&1\end{smallmatrix}\right)),\,(0)).$$
    The matrices $\left(\begin{smallmatrix}1&1\\1&0\end{smallmatrix}\right)$ and $\left(\begin{smallmatrix}0&1\\1&1\end{smallmatrix}\right)$ are conjugate by a matrix $A$. The two meshes are homologous, with $\psi_1(x)=Ax$ and $d_1=0$.
  \item Consider two meshes $$((\Z_3,\Z_1),\,\left(\begin{smallmatrix}0&0\\0&0\end{smallmatrix}\right),\,\left(\begin{smallmatrix}0&0\\1&0\end{smallmatrix}\right))
\quad\text{ and }\quad ((\Z_3,\Z_1),\,\left(\begin{smallmatrix}0&0\\0&0\end{smallmatrix}\right),\,\left(\begin{smallmatrix}0&0\\2&0\end{smallmatrix}\right)).$$
The two meshes are homologous, with $\pi=id$, $\psi_1(x)=-x$, $\psi_2=id$ and $d_1=d_2=0$.
    \item Consider two meshes
$$((\Z_2,\Z_1,\Z_1),\,\left(\begin{smallmatrix}0&0&0\\0&0&0\\0&0&0\end{smallmatrix}\right),\,\left(\begin{smallmatrix}0&0&0\\1&0&0\\0&0&0\end{smallmatrix}\right))
\quad\text{ and }\quad
((\Z_2,\Z_1,\Z_1),\,\left(\begin{smallmatrix}0&0&0\\0&0&0\\0&0&0\end{smallmatrix}\right),\,\left(\begin{smallmatrix}0&0&0\\0&0&0\\1&0&0\end{smallmatrix}\right)).$$
The two meshes are homologous, with $\pi=(2\ 3)$, $\psi_1=\psi_2=\psi_3=id$ and $d_1=d_2=d_3=0$.
\end{itemize}
\end{example}

The next example shows that, in the definition of homologous meshes, we have to consider the constants $d_i$.

\begin{example}
Consider two meshes
$$((\Z_3,\Z_3),\,\left(\begin{smallmatrix}2&1\\1&2\end{smallmatrix}\right),\,\left(\begin{smallmatrix}0&0\\0&0\end{smallmatrix}\right))
\quad\text{ and }\quad
((\Z_3,\Z_3),\,\left(\begin{smallmatrix}2&1\\1&2\end{smallmatrix}\right),\,\left(\begin{smallmatrix}0&1\\1&0\end{smallmatrix}\right)).$$
To show that the two meshes are homologous, without loss of generality put $\pi=id$ (due to symmetry). Condition (H1) for $i=1$, $j=2$ implies that $\psi_1=\psi_2$. Condition (H2) for $i=1$, $j=2$ says that $\psi_2(0)=1+d_1-2d_2$, hence we cannot have both $d_1=d_2=0$. One can check that $\psi_1=\psi_2=id$, $d_1=0$, $d_2=2$ satisfies all conditions.
\end{example}

\begin{remark}\label{rem:iso_thm}
Homology of affine meshes can be restated in terms of a group action. Let $A_j$, $j\in J$, be pairwise non-isomorphic abelian groups and $n_j$, $j\in J$, cardinal numbers. Consider the set $X$ of all indecomposable affine meshes with $n_j$ fibres equal to $A_j$. Formally, $X$ consists of all meshes
$(B_i;\,\varphi_{i,j};\,c_{i,j})$ over the index set $I=\sum n_j$ such that the tuple $(B_i:i\in I)$ is obtained from $(A_j:j\in J)$ by replacing each $A_j$ with $n_j$ copies of itself. Then two meshes $\mathcal A=(B_i;\,\varphi_{i,j};\,c_{i,j})$ and $\mathcal A'=(B_i;\,\varphi_{i,j}';\,c_{i,j}')$ are homologous if and only if $g(\mathcal A)=\mathcal A'$ for some $g\in G$, where
$$G=\prod_{j\in J} (A_j\rtimes\aut{A_j})\wr S_{n_j}=\left(\prod_{i\in I} (B_i\rtimes\aut{B_i})\right)\rtimes S$$
where $S$ contains all permutations $\pi\in S_I$ such that $\pi(B_i)\simeq B_i$ (in particular, $S\simeq\prod_{j\in J} S_{n_j}$). The action of an element $g=(\bar d,\bar\psi,\pi)\in G$ on $X$ is defined by
$$g(B_i;\,\varphi_{i,j};\,c_{i,j})=\left(B_i;\,\psi_j^{-1}\varphi_{\pi i,\pi j}\psi_i;\,\psi_j^{-1}(c_{\pi i,\pi j})+\psi_j^{-1}\varphi_{\pi i,\pi j}(d_i)-\psi_j^{-1}\varphi_{\pi j,\pi j}(d_j)\right).$$
This interpretation of homology will be useful in the enumeration of medial quandles in Section~\ref{sec:enumeration}.
\end{remark}

\section{Latin orbits}\label{sec:latin_orbits}

The orbits in a medial quandle need not be algebraically connected (as subquandles). In this section, we investigate the ``most structural" case, when all orbits are latin, while the next section partly addresses the ``structureless" case, when all orbits are projection quandles.

It follows from Proposition \ref{prop:gcd} that in a medial quandle, only the smallest orbits can be latin, and only if their size divides the size of any other orbit. In particular, if all orbits are latin, then they have equal size.
The highlight of this section is a somewhat surprising Theorem \ref{thm:latin_orbits} saying that all such quandles are direct products of a latin quandle and a projection quandle. For finite quandles, we get a stronger statement that can be rephrased in the following way: every finite latin medial quandle $Q$ can be extended uniquely to a medial quandle with a given number of orbits of size $|Q|$.

We start with two important observations on medial quandles with latin orbits. Notice that an orbit $Qe$ is latin if and only if, in the canonical mesh of $Q$ over a transversal $E$ containing $e$, the mapping $\varphi_{e,e}$ is a permutation.

\begin{prop}\label{p:finite_latin_orbit}
Consider a medial quandle such that all orbits have equal finite size and one of them is latin (as a subquandle). Then all orbits are latin.
\end{prop}

\begin{proof}
Consider the canonical mesh of such a quandle $Q$ over a transversal $E$ containing $e$, let $Qe$ be a latin orbit. Then $\varphi_{e,e}$ is a permutation. Consider an arbitrary $f\in E$. By (M3), we have $\varphi_{e,e}^2=\varphi_{f,e}\varphi_{e,f}$, hence the mapping $\varphi_{e,f}$ is 1-1 and $\varphi_{f,e}$ is onto. But all orbits have equal finite size, hence both $\varphi_{e,f},\varphi_{f,e}$ are bijections, and so is $\varphi_{f,f}$, because $\varphi_{f,f}^2=\varphi_{e,f}\varphi_{f,e}$ by (M3). Hence all orbits are latin.
\end{proof}

\begin{prop}\label{p:latin_orbits}
Consider a medial quandle such that all orbits are latin. Then
\begin{enumerate}
    \item all orbit groups are isomorphic;
    \item all orbits are isomorphic as quandles.
\end{enumerate}
\end{prop}

\begin{proof}
Consider a canonical mesh of such a quandle $Q$. All mappings $\varphi_{e,e}$ are permutations. By (M3), we have $\varphi_{e,e}^2=\varphi_{f,e}\varphi_{e,f}$ for every $e,f\in E$, hence all mappings $\varphi_{e,f}$ are permutations, and thus isomorphisms $\orb Qe\simeq\orb Qf$. By (M3) again, we have $\varphi_{f,f}\varphi_{e,f}=\varphi_{e,f}\varphi_{e,e}$, hence $\varphi_{f,f}=\varphi_{e,e}^{\varphi_{e,f}}$, and according to Corollary \ref{c:iso_conn}, the orbits $Qe$ and $Qf$ are isomorphic (as affine quandles).
\end{proof}

An interesting consequence is that in any medial quandle, all latin orbits are isomorphic: consider the subquandle of all elements that belong to a latin orbit.

Now we show two technical lemmas on affine meshes that result in quandles with latin orbits. First, we show that, up to isomorphism, we can always take the constant matrix zero. Next, we show that, up to isomorphism, there is only one choice of the homomorphism matrix.
Without loss of generality, we shall consider all orbit groups equal.

\begin{lemma}\label{l:latin_orbits1}
Let $\mathcal A=((A,A,\dots);\varphi_{i,j};c_{i,j})$ be an indecomposable affine mesh over a set $I$ such that $\varphi_{i,i}$ is a permutation for every $i\in I$. Then the sum of $\A$ is isomorphic to the sum of the affine mesh $\mathcal A'=((A,A,\dots);\varphi_{i,j};0)$.
\end{lemma}

\begin{proof}
First observe that $\mathcal A'$ is an indecomposable affine mesh, because all mappings $\varphi_{i,i}$ are onto $A$. So we can use Theorem \ref{thm:isomorphism}. Let for every $i\in I$ $$\pi=id,\quad \psi_i=id, \quad d_i=-\varphi_{i,i}^{-1}(c_{1,i}).$$
Condition (H1) is satisfied trivially, we check (H2). Since $\varphi_{i,j}'=\varphi_{i,j}$ and $c_{i,j}'=0$, we need to check that $$c_{i,j}=\varphi_{i,j}(d_i)-\varphi_{j,j}(d_j)=\varphi_{i,j}(d_i)+c_{1,j}.$$
Using the definition of $d_i$ again, we obtain
$$\varphi_{j,j}\varphi_{i,j}(d_i)\stackrel{(M3)}{=} \varphi_{i,j}\varphi_{i,i}(d_i)=-\varphi_{i,j}(c_{1,i})\stackrel{(M4)}{=} -\varphi_{j,j}(c_{1,j}-c_{i,j}).$$
Since $\varphi_{j,j}$ is bijective, we obtain $\varphi_{i,j}(d_i)=c_{i,j}-c_{1,j}$, as required.
\end{proof}

\begin{lemma}\label{l:latin_orbits2}
Let $\mathcal A=((A,A,\dots);\varphi_{i,j};0)$ be an indecomposable affine mesh over a set $I$ such that $\varphi_{i,i}$ is a permutation for every $i\in I$.
Then the sum of $\A$ is isomorphic to the sum of the affine mesh $\mathcal A'=((A,A,\dots);\varphi_{i,j}';0)$ with $\varphi_{i,j}'=\varphi_{1,1}$ for every $i,j$.
\end{lemma}

\begin{proof}
First observe that $\mathcal A'$ is an indecomposable affine mesh, because $\varphi_{1,1}$ is onto $A$. So we can use Theorem \ref{thm:isomorphism}. Let for every $i\in I$ $$\pi=id,\quad \psi_i=\varphi_{i,1}, \quad d_i=0.$$
All mappings $\psi_i$ are bijective, because $\varphi_{i,i}^2=\varphi_{1,i}\varphi_{i,1}$ and $\varphi_{1,1}^2=\varphi_{i,1}\varphi_{1,i}$
according to (M3).
Condition (H1), with $\varphi_{i,j}'=\varphi_{1,1}$, states $$\varphi_{j,1}\varphi_{i,j}=\varphi_{1,1}\varphi_{i,1},$$
which is a special case of condition (M3) on $\mathcal A$.
Condition (H2) is satisfied trivially.
\end{proof}

Notice that the mesh $\mathcal A'$ in the previous lemma describes the direct product $\Aff(A,1-\varphi_{1,1})\times P$ where $P$ is a projection quandle over $I$.
The main result of this section follows easily.

\begin{theorem}\label{thm:latin_orbits}
Consider a medial quandle such that all orbits are latin. Then it is isomorphic to a direct product of a latin quandle and a projection quandle.
\end{theorem}

\begin{proof}
Denote $Q$ such a quandle and let $Q_0$ be one of its orbits. Its canonical mesh satisfies the assumptions of Lemmas \ref{l:latin_orbits1} and \ref{l:latin_orbits2}, hence $Q$ is isomorphic to $Q_0\times P$, where $P$ is a projection quandle over the set of orbits.
\end{proof}

Using Proposition \ref{p:finite_latin_orbit}, we immediately obtain the following.

\begin{corollary}\label{thm:finite_latin_orbit}
Consider a medial quandle such that all orbits have equal finite size and one of them is latin (as a subquandle). Then it is isomorphic to a direct product of a latin quandle and a projection quandle.
\end{corollary}

\begin{example}
Consider a medial quandle $Q$ with $m$ orbits of prime size $p$. According to Proposition \ref{p:finite_latin_orbit}, there are two essentially different types of such quandles.
\begin{enumerate}
    \item All orbits are latin. Then $Q$ is isomorphic to $\Aff(\Z_p,f)\times P$, where $f\in\{2,\dots,p-1\}$ and $P$ is a projection quandle of size $m$. There are $p-2$ such quandles up to isomorphism.
    \item None of the orbits is latin. Then all orbits are isomorphic to $\Aff(\Z_p,1)$, hence are projection quandles. We shall see later in Example \ref{ex:orbits_Zp} that there are at least $$p^{m(m-p(1+\log_pm)-2)}$$ such quandles up to isomorphism. For $p$ fixed, the growth rate is at least $p^{m^2-O(m\log m)}$.
\end{enumerate}
\end{example}

Quandles where all orbits are projection quandles will be called \emph{2-reductive} and studied in the next section.

\section{Reductivity}\label{sec:reductivity}

A binary algebra $Q$ is called (left) $m$-\emph{reductive}, if $(R_y)^m$ is a constant mapping onto $y$, i.e., if it satisfies the identity
\begin{equation*}
(((x\underbrace{y)y)\ldots )y}_{m-\text{times}}=y
\end{equation*}
for every $x,y\in Q$. If $Q$ is medial and idempotent, this identity is equivalent to a more general condition that any composition $R_{z_1}R_{z_2}\cdots R_{z_m}$ is a constant mapping, i.e.,
\begin{equation*}
(((xz_1)z_2)\ldots )z_m=(((yz_1)z_2)\ldots )z_m
\end{equation*}
for every $x,y,z_1,\dots,z_m\in Q$, see \cite[Lemma 1.2]{PR}. An binary algebra will be called \emph{reductive}, if it is $m$-reductive for some $m$.
The phenomenon of $m$-reductivity in the general context of medial idempotent binary algebras was studied in \cite{PR}, the special but very important case $m=2$ in greater detail in \cite{RS91} (under the name \emph{differential groupoids}), and a generalization to higher arities in \cite{KPRS}.

Let $Q=\aff{A,f}$ be an affine quandle. It is easy to calculate
$$(((x\underbrace{y)y)\ldots )y}_{m-\text{times}}=(1-f)^m(x)+(1-(1-f)^m)(y),$$
hence $Q$ is $m$-reductive if and only if $(1-f)^m=0$.

\begin{exm}
Let $p^m$ be a prime power. Then $\aff{\Z_{p^m},1-p}$ is an $m$-reductive medial quandle which is not $n$-reductive for any $n<m$.
\end{exm}

We show that the orbits of an $m$-reductive medial quandle satisfy the more restrictive condition $(1-f)^{m-1}=0$.
The same property actually characterizes the affine meshes that result in $m$-reductive quandles.

\begin{prop}\label{prop:reductive}
Let $\mathcal A=(A_i;\,\varphi_{i,j};\,c_{i,j})$ be an indecomposable affine mesh over a set $I$. Then the sum of $\mathcal A$ is $m$-reductive if and only if, for every $i\in I$, $$\varphi_{i,i}^{m-1}=0.$$
\end{prop}

\begin{proof}
Let $Q$ be the sum of the mesh $\mathcal A$. Then, for every $a\in A_i$ and $b\in A_j$,
\begin{equation}\label{eq:mred}
(((a\underbrace{b)b)\ldots )b}_{m-\text{times}}=\varphi_{j,j}^{m-1}(c_{i,j}+\varphi_{i,j}(a))\,+\,\sum_{r=0}^{m-1}\varphi_{j,j}^{r}(1-\varphi_{j,j})(b).
\end{equation}

($\Rightarrow$)
Assuming $m$-reductivity, expression \eqref{eq:mred} equals $b$, and taking $b=0$ in the group $A_j$, we obtain
$$\varphi_{j,j}^{m-1}(c_{i,j}+\varphi_{i,j}(a))=0.$$
Indecomposability of the mesh means that
$$A_j=\langle c_{i,j}+\varphi_{i,j}(a):\ i\in I,\ a\in A_i\rangle,$$
hence
$$\varphi_{j,j}^{m-1}(x)=0$$ for every $x\in A_j$.

($\Leftarrow$)
In view of \eqref{eq:mred}, we need to show that $$\sum_{r=0}^{m-1}\varphi_{j,j}^{r}(1-\varphi_{j,j})(b)=b.$$
The sum telescopes, we obtain $\sum_{r=0}^{m-1}\varphi_{j,j}^{r}(1-\varphi_{j,j})=\sum_{r=0}^{m-1}(\varphi_{j,j}^{r}-\varphi_{j,j}^{r+1})=1-\varphi_{j,j}^m=1$.
\end{proof}

\begin{corollary}\label{cor:copr3}
Let $Q$ be a medial quandle. If the orbit sizes are coprime, then $Q$ is 3-reductive.
\end{corollary}

\begin{proof}
Assume $Q$ is the sum of an indecomposable affine mesh $(A_i;\,\varphi_{i,j};\,c_{i,j})$ over a set $I$.
Proposition \ref{prop:gcd} implies that, for every $i\in I$, $|\im{\varphi_{i,i}^2}|=1$, hence $\varphi_{i,i}^2=0$, and $Q$ is 3-reductive by Proposition \ref{prop:reductive}.
\end{proof}

We proceed with an interesting observation: if one of the diagonal homomorphisms is nilpotent, then all diagonal homomorphisms are nilpotent.

\begin{lemma}\label{lem:reductive}
Let $\mathcal A=(A_i;\,\varphi_{i,j};\,c_{i,j})$ be an affine mesh over a set $I$ such that $\varphi_{i,i}^m=0$ for some $i\in I$. Then $\varphi_{j,j}^{m+2}=0$ for every $j\in I$.
\end{lemma}

\begin{proof}
Applying (M3) $(m+1)$-times, we see that $\varphi_{j,j}^{m+2}=\varphi_{i,j}\varphi_{i,i}^m\varphi_{j,i}=0$ for every $j\in I$.
\end{proof}

\begin{example}
Orbits (considered as subquandles) may have different degrees of reductivity. For (the smallest) example, consider the mesh
$$((\Z_4,\Z_2),\,\left(\begin{smallmatrix}2&0\\2&0\end{smallmatrix}\right),\,\left(\begin{smallmatrix}0&1\\1&0\end{smallmatrix}\right)).$$
The first orbit is 2-reductive, but not 1-reductive. The second orbit is 1-reductive.
\end{example}

As a consequence of the observation, we obtain the following characterization of reductive medial quandles.

\begin{theorem}\label{thm:reductive}
Let $Q$ be a medial quandle. Then the following statements are equivalent.
\begin{enumerate}
    \item $Q$ is reductive.
    \item At least one orbit of $Q$ is reductive.
    \item All orbits of $Q$ are reductive.
\end{enumerate}
Moreover,
\begin{enumerate}
    \item[(a)] $Q$ is $m$-reductive if and only if all orbits of $Q$ are $(m-1)$-reductive;
    \item[(b)] if one orbit of $Q$ is $m$-reductive, then $Q$ is $(m+3)$-reductive.
\end{enumerate}
\end{theorem}

\begin{proof}
Assume $Q$ is the sum of an indecomposable affine mesh $(A_i;\,\varphi_{i,j};\,c_{i,j})$ over a set $I$. An orbit $A_i$, as an affine quandle, is $m$-reductive if and only if $\varphi_{i,i}^{m}=0$. Hence, statement (a) is essentially Proposition \ref{prop:reductive}, and statement (b) follows from (a) using Lemma \ref{lem:reductive}. The equivalence of conditions (1), (2), (3) follows immediately.
\end{proof}

\begin{example}
Let $Q$ be a medial quandle such that one of its orbit groups is isomorphic to $\Z_{2^m}$. Then $Q$ is $(m+3)$-reductive, because for every $f\in\aut{\Z_{2^m}}$, we have $(1-f)^m=0$, hence one orbit of $Q$ is $m$-reductive and Theorem \ref{thm:reductive} applies.
\end{example}

The 2-reductive case is of particular interest (see Section
\ref{sec:enumeration}). Proposition \ref{prop:reductive} says that a
medial quandle is 2-reductive if and only if every orbit is a
projection quandle (the condition $\varphi_{i,i}=0$ means that
the orbit is $\aff{A,1}$). With a little extra work, we obtain a
stronger representation theorem. We start with a lemma stating that,
in the homomorphism matrix, zeros propagate vertically, i.e., if a
column contains zero, the whole column is zero.

\begin{lm}\label{lem:2-reductive}
Let $\mathcal A=(A_i;\,\varphi_{i,j};\,c_{i,j})$ be an indecomposable affine mesh over a set $I$. Assume there are $j,k\in I$ such that $\varphi_{j,k}=0$. Then $\varphi_{i,k}=0$ for every $i\in I$.
\end{lm}

\begin{proof}
First, we show that $\varphi_{k,k}=0$. The indecomposability condition says that $A_k=\langle c_{i,k}+\im(\varphi_{i,k}):i\in I\rangle$, so it is sufficient to verify that $\varphi_{k,k}\varphi_{i,k}=0$ and $\varphi_{k,k}(c_{i,k})=0$ for every $i\in I$.
By (M3),
\begin{equation*}
\varphi_{k,k}\varphi_{i,k}=\varphi_{j,k}\varphi_{i,j}=0
\end{equation*}
for every $i\in I$, because $\varphi_{j,k}=0$ by the assumptions.
Similarly, by (M4),
\begin{equation*}
0=\varphi_{j,k}(c_{i,j})=\varphi_{k,k}(c_{i,k}-c_{j,k}),
\end{equation*}
and thus
\begin{equation*}
\varphi_{k,k}(c_{i,k})=\varphi_{k,k}(c_{j,k}),
\end{equation*}
for every $i\in I$.
With $i=k$, we see that $\varphi_{k,k}(c_{j,k})=0$, and thus $\varphi_{k,k}(c_{i,k})=0$ for every $i\in I$. Hence $\varphi_{k,k}=0$.

In the second step, fix $i\in I$, and we show that $\varphi_{i,k}=0$. Again, since $A_i=\langle c_{l,i}+\im(\varphi_{l,i}):l\in I\rangle$,
it is sufficient to verify that $\varphi_{i,k}\varphi_{l,i}=0$ and $\varphi_{i,k}(c_{l,i})=0$ for every $l\in I$.
By (M3),
\begin{equation*}
\varphi_{i,k}\varphi_{l,i}=\varphi_{k,k}\varphi_{l,k}=0
\end{equation*}
for every $l\in I$, using $\varphi_{k,k}=0$.
Similarly, by (M4),
\begin{equation*}
\varphi_{i,k}(c_{l,i})=\varphi_{k,k}(c_{l,k}-c_{i,k})=0
\end{equation*}
for every $l\in I$. Hence $\varphi_{i,k}=0$.
\end{proof}

Now we can prove the characterization of affine meshes that result in 2-reductive medial quandles.
The equivalence of (1) and (3) was proved by Roszkowska and Romanowska in \cite[Section 2]{RR89}.

\begin{theorem}\label{thm:2-reductive}
Let $Q$ be a medial quandle and assume it is the sum of an indecomposable affine mesh $(A_i;\,\varphi_{i,j};\,c_{i,j})$ over a set $I$. Then the following statements are equivalent.
\begin{enumerate}
    \item $Q$ is 2-reductive.
    \item For every $j\in I$, there is $i\in I$ such that $\varphi_{i,j}=0$.
    \item $\varphi_{i,j}=0$ for every $i,j\in I$.
\end{enumerate}
\end{theorem}

\begin{proof}
(3) $\Rightarrow$ (1) $\Rightarrow$ (2) follows from Proposition \ref{prop:reductive}. (2) $\Rightarrow$ (3) follows from Lemma \ref{lem:2-reductive}.
\end{proof}

\begin{corollary}\label{cor:copr}
Let $Q$ be a medial quandle with finite orbits and assume that for every orbit $A$ there is an orbit $B$ such that $|A|$ and $|B|$ are coprime. Then $Q$ is 2-reductive.
\end{corollary}

\begin{proof}
The condition implies that, in a corresponding affine mesh, for every $j$, there is $i$ such that $\varphi_{i,j}=0$, hence Theorem \ref{thm:2-reductive} applies.
\end{proof}

In particular, medial quandles with a one-element orbit are always 2-reductive.

The isomorphism theorem for 2-reductive medial quandles is significantly simpler than the general Theorem \ref{thm:isomorphism}, because the homomorphism matrices are trivial.

\begin{theorem}\label{thm:2-reductive_iso}
Let $\mathcal A=(A_i;\,0;\,c_{i,j})$ and $\mathcal A'=(A_i';\,0;\,c_{i,j}')$ be two indecomposable affine meshes, over the same index set $I$.
Then the sums of $\A$ and $\A'$ are isomorphic quandles if and only if there is $\pi\in S_n$ and $\psi_i:A_i\simeq A_{\pi i}'$ such that $\psi_j(c_{i,j})=c_{\pi i,\pi j}'$.
\end{theorem}

\begin{proof}
This is a special case of Theorem \ref{thm:isomorphism}.
Since $\varphi_{i,j}=0$ and $\varphi_{i,j}'=0$, condition (H1) is trivial, and condition (H2) is satisfied regardless the values of the constants $d_i$.
\end{proof}

\begin{example}
Up to isomorphism, there is precisely one medial quandle $Q$ with two orbits of given coprime size. According to Corollary
\ref{cor:copr}, $Q$ is 2-reductive, hence it is the sum of an indecomposable mesh
$$((A,B),\,\left(\begin{smallmatrix}0&0\\0&0\end{smallmatrix}\right),\,\left(\begin{smallmatrix}0&b\\a&0\end{smallmatrix}\right)).$$
Indecomposability implies that $A=\langle a\rangle$ and $B=\langle b\rangle$, hence the groups are cyclic, and according to Theorem
\ref{thm:2-reductive_iso}, all choices of $a,b$ result in isomorphic quandles.
\end{example}

\begin{example} \label{ex:orbits_Zp}
Consider all 2-reductive medial quandles with $m$ orbits, all of a prime size $p$. They are given by indecomposable affine meshes of the form $((\Z_p,\dots,\Z_p);0;c_{i,j})$ over the set $\{1,\dots,m\}$, i.e., by $m\times m$ matrices over $\Z_p$ with zero diagonal such that all columns are non-zero (and thus generate the group $\Z_p$).  There are precisely $(p^{m-1}-1)^m$ such matrices. Since every quandle with $n=pm$ elements is isomorphic to at most $n!$ quandles, the number of isomorphism classes is at least
$$\frac{(p^{m-1}-1)^m}{(pm)!}\geq\frac{p^{(m-2)m}}{(pm)^{pm}}=p^{m^2-2m-(1+\log_pm)pm}.$$
\end{example}

We see that 2-reductive medial quandles have a rather combinatorial character: they are constructed from any tuple of abelian groups and an arbitrary matrix of constants with zero diagonal and columns generating the respective fibres. The operation is rather simplistic, $$a*b=b+c_{i,j},$$ for every $a\in A_i$ and $b\in A_j$. An isomorphism between quandles is given by isomorphisms between the fibres preserving the constants. This informally explains the combinatorial explosion in the number of 2-reductive medial quandles constructed in Example \ref{ex:orbits_Zp}, and also witnessed by computation in Section \ref{ssec:computation}.
In contrast, our computation results suggest that non-2-reductive medial quandles are fairly rare.

\section{Symmetry}\label{sec:symmetry}

A binary algebra $Q$ is called (left) $n$-\emph{symmetric}, if $(L_a)^n=1$ for every $a\in Q$, i.e., if it satisfies the identity
\begin{equation*}
\underbrace{x(x(\ldots(x}_{n-\text{times}}y)))=y.
\end{equation*}
for every $x,y\in Q$. Note that 2-symmetry is just another name for being \emph{involutory}. (The term ``symmetric" is somewhat misleading, nevertheless widely used in papers on binary algebras. Involutory quandles are also called \emph{keis} in some papers.)

Involutory medial quandles were investigated by Roszkowska in great
detail in the aforementioned series \cite{R87,R89,R99a,R99b}. The
first and the second papers contain a syntactic analysis, resulting
in the description of all varieties (equational theories) of
involutory medial quandles. The third paper develops a structure
theory; the main result, \cite[Theorem 4.3]{R99a}, is obtained in
the present section as Corollary \ref{cor:2-symmetric}. The last
paper contains the classification of subdirectly irreducible
involutory medial quandles, see the discussion in Section
\ref{sec:congruences}.

Let $Q=\aff{A,f}$ be an affine quandle. It is easy to calculate
$$\underbrace{x(x(\ldots(x}_{n-\text{times}}y)))=(1-f^n)(x)+f^n(y),$$
hence $Q$ is $n$-symmetric if and only if $f^n=1$.

\begin{exm}
Let $F$ be a field and $r$ a primitive $n$-th root of unity. Then
$\aff{F,r}$ is an $n$-symmetric medial quandle which is not
$m$-symmetric for any $m<n$. For example, we can take $F=\mathbb C$
and $r=e^{2\pi i/n}$, or we can take $F=\Z_p$ with $p$ prime
and $n\mid p-1$.
\end{exm}

Notice that $1-f^n=(1-f)\cdot\sum_{i=0}^{n-1}f^i$. If the sum is zero, then $\aff{A,f}$ is $n$-symmetric. The converse is not true in general, e.g., for $A=\Z_{15}$ and $f=11$ we have $f^2=1$ and $f\neq\pm1$. Our next result implies that the orbits of $n$-symmetric medial quandles can always be represented as $\aff{A,f}$ with $f\in\aut A$ satisfying $\sum_{i=0}^{n-1}f^i=0$. Similarly to the reductive case, this is the property that charaterizes the affine meshes that result in $n$-symmetric quandles.

\begin{prop}\label{prop:symmetric}
Let $\mathcal A=(A_i;\,\varphi_{i,j};\,c_{i,j})$ be an indecomposable affine mesh over a set $I$. Then the sum of $\mathcal A$ is $n$-symmetric if and only if, for every $i\in I$, $$\sum_{r=0}^{n-1}(1-\varphi_{i,i})^r=0.$$
\end{prop}

Recall that every orbit $A_i$, as a subquandle, equals $\aff{A_i,1-\varphi_{i,i}}$. This justifies the claim above Proposition \ref{prop:symmetric}.

\begin{proof}
Let $Q$ be the sum of the mesh $\mathcal A$. Then, for every $a\in A_i$ and $b\in A_j$,
\begin{equation}\label{eq:nsym}
\underbrace{a(a(\ldots(a}_{n-\text{times}}b)))=\left(\sum_{r=0}^{n-1}(1-\varphi_{j,j})^r\right)(c_{i,j}+\varphi_{i,j}(a))\,+\,(1-\varphi_{j,j})^n(b).
\end{equation}

($\Rightarrow$)
Assuming $n$-symmetry, expression \eqref{eq:nsym} equals $b$, and taking $b=0$ in the group $A_j$, we obtain
$$\left(\sum_{r=0}^{n-1}(1-\varphi_{j,j})^r\right)(c_{i,j}+\varphi_{i,j}(a))=0.$$
Indecomposability of the mesh means that
$$A_j=\langle c_{i,j}+\varphi_{i,j}(a):\ i\in I,\ a\in A_i\rangle,$$
hence
$$\left(\sum_{r=0}^{n-1}(1-\varphi_{j,j})^r\right)(x)=0$$ for every $x\in A_j$.

($\Leftarrow$)
Put $f_i=1-\varphi_{i,i}$ for every $i\in I$. The assumption says that $\sum_{r=0}^{n-1}f_i^r=0$, hence also $1-f_i^n=(1-f_i)(\sum_{r=0}^{n-1}f_i^r)=0$, and thus $(1-\varphi_{i,i})^n=f_i^n=1$, for every $i\in I$.  The $n$-symmetric law follows immediately from \eqref{eq:nsym}.
\end{proof}

As a special case, we obtain Roszkowska's representation theorem for involutory medial quandles \cite[Theorem 4.3]{R99a}.
(Roszkowska uses a slightly different notation: the translation between her mappings $h^i_j:A_i\to A_j$ and our parameters is: $h^i_j(a)=\varphi_{i,j}(a)+c_{i,j}$ in one direction, and $\varphi_{i,j}(a)=h^i_j(a)-h^i_j(0)$, $c_{i,j}=h^i_j(0)$ in the other.)

\begin{corollary}\label{cor:2-symmetric}
A binary algebra is an involutory medial quandle if and only if it is the sum of an indecomposable affine mesh $\mathcal A=(A_i;\,\varphi_{i,j};\,c_{i,j})$ over a set $I$ where $\varphi_{i,i}=2$ for every $i\in I$.
\end{corollary}

\begin{proof}
Theorem \ref{thm:decomposition} and Proposition \ref{prop:symmetric} say that involutory (i.e., 2-symmetric) medial quandles are precisely the sums of indecomposable affine meshes satisfying $(1-\varphi_{i,i})^0+(1-\varphi_{i,i})^1=2-\varphi_{i,i}=0$ for every $i\in I$.
\end{proof}

Affine quandles of the form $\aff{A,-1}$ are called \emph{dihedral quandles} \cite{Car}, or \emph{cores} of abelian groups \cite{R87}.
Corollary \ref{cor:2-symmetric} can be restated as follows.

\begin{corollary}\label{cor:2-symmetric*}
Let $Q$ be a medial quandle. Then $Q$ is involutory if and only if all orbits are dihedral quandles (cores of abelian groups).
\end{corollary}

We finish the section with remarks on medial quandles that are reductive and symmetric at the same time.

\begin{example}
Let $m$ be a natural number, $p>m$ a prime and let $Q=\aff{(\Z_p)^m,f}$ where
$$f=\left(\begin{matrix} 1&1&\dots&0&0\\&&\dots&&\\0&0&\dots&1&1\\0&0&\dots&0&1\end{matrix}\right)$$
is a Jordan matrix. It is not difficult to calculate that $Q$ is $p$-symmetric, but not $i$-symmetric for any $i<p$, and it is $m$-reductive, but not $i$-reductive for any $i<m$.
\end{example}

As an immediate corollary to our Propositions \ref{prop:reductive} and \ref{prop:symmetric}, we also obtain \cite[Proposition 2.2]{RR89}: in 2-reductive $n$-symmetric medial quandles, the orbit groups have exponent dividing $n$. Indeed, from 2-reductivity we get $\varphi_{i,i}=0$, and $n$-symmetry forces $0=\sum_{r=0}^{n-1}(1-\varphi_{i,i})^r=n$ in every orbit.
Such quandles were called \emph{$n$-cyclic groupoids} in \cite{Plo,RR89}. The first paper contains a description of free $n$-cyclic groupoids, and of subdirectly irreducible $n$-cyclic groupoids for $n$ prime. The second paper develops a structural theorem we described in Theorem \ref{thm:2-reductive}, and, using this representation, they describe congruences and subdirectly irreducible algebras for arbitrary $n$ (for the statement, see also our Theorem \ref{thm:red_si}).

The dual case, $m$-reductive involutory (i.e., 2-symmetric) medial quandles, is also interesting, although we could not find any explicit reference in literature. An analogous argument leads to the conclusion that the orbit groups have exponent dividing $2^{m-1}$, because $\varphi_{i,i}=2$, and thus $2^{m-1}=0$ by $m$-reductivity.

We also mention that \cite[Section 5]{PRR96} contains some independence results concerning the varieties of $n$-symmetric $m$-reductive medial quandles, their duals and latin medial quandles.

The whole story of symmetric and reductive binary algebras can be traced back to 1970's when mathematicians searched for equational theories with very few term operations. The variety of 2-reductive involutory medial quandles has precisely $n$ essentially $n$-ary term operations \cite{Pl71}.

\section{Enumerating medial quandles}\label{sec:enumeration}

\subsection{Asymptotic results}\label{ssec:asymptotics}

Blackburn \cite{B} proved that the number of isomorphism classes of quandles of
order $n$ grows as $2^{\Theta(n^2)}$ (we recall that $f=\Theta(g)$ if $f=O(g)$ and $g=O(f)$). For the lower bound, he
provides a construction of $2^{\frac14n^2-O(n\log n)}$ involutory
quandles. His construction is essentially a special case of Example
\ref{ex:orbits_Zp}, with $p=2$. Since such quandles are 2-reductive,
we can refine Blackburn's statement of \cite[Theorem 11]{B}.

\begin{theorem}
The number of isomorphism classes of 2-reductive involutory medial quandles of order $n$ is at least $2^{\frac14n^2-O(n\log n)}$.
\end{theorem}

\begin{proof}
Let $n$ be even. All affine meshes of the form $((\Z_2,\dots,\Z_2);\,0;(c_{i,j}))$ with $n/2$ fibres result in 2-reductive involutory medial quandles (see Theorem \ref{thm:2-reductive}, Corollary \ref{cor:2-symmetric} and notice that $0=2$). In Example \ref{ex:orbits_Zp}, we calculated that such meshes result in at least
$$2^{(\frac n2)^2-2(\frac n2)-(1+\log\frac n2)n}=2^{\frac14n^2-O(n\log n)}$$
pairwise non-isomorphic quandles.
For $n$ odd, consider an additional fibre $\Z_1$ and obtain the same estimate.
\end{proof}

Using our theory, it is not difficult to prove a tight upper bound for 2-reductive medial quandles.

\begin{theorem}
The number of isomorphism classes of 2-reductive medial quandles of order $n$ is at most $2^{(\frac14+o(1))n^2}$.
\end{theorem}

\begin{proof}
Using Theorem \ref{thm:2-reductive}, an upper bound on the number of 2-reductive medial quandles of size $n$ can be calculated the following way: for each partition $n=n_1+\ldots+n_k$, and for each choice of $n_i$-element abelian groups, count the number of $k\times k$ matrices where the entry at the position $(i,j)$, $i\neq j$ comes from the $n_j$-element group, while the diagonal entries are zero (not all choices result in an indecomposable mesh, but this is irrelevant for the upper bound).

The number of isomorphism classes of $m$-element abelian groups is certainly at most $m$. Using this estimate, there are at most $n_1\cdot\ldots\cdot n_k\cdot n_1^{k-1}\cdot\ldots\cdot n_k^{k-1}=(n_1\cdot\ldots\cdot n_k)^k$ isomorphism classes of 2-reductive medial quandles with given partition $n=n_1+\ldots+n_k$. An easy argument shows that the maximal value of $(n_1\cdot\ldots\cdot n_k)^k$, over all partitions of $n$, happens when $n_1=\ldots=n_{n/2}=2$ for $n$ even, and $n_1=\ldots=n_{(n-1)/2}=2$, $n_{(n+1)/2}=1$ for $n$ odd (sketch of the proof: first notice that replacing $n_i>3$ by $n_i-2,2$ increases the value, hence only $n_i\in\{1,2,3\}$ can maximize the expression; then it is easy to calculate that $1,3\to 2,2$ increases the value, hence either all $n_i\in\{1,2\}$, or all $n_i\in\{2,3\}$; in the former case, $1,1\to 2$ increases the value; in the latter case, $3\to 2,1$ increases the value). In either case, the maximal value is
$2^{\lfloor \frac14n^2\rfloor}$. The number of partitions of $n$ is asymptotically $2^{\Theta(\sqrt n)}$, hence there are at most $2^{\Theta(\sqrt n)}\cdot 2^{\lfloor \frac14n^2\rfloor}=2^{(\frac14+o(1))n^2}$ isomorphism classes of 2-reductive medial quandles.
\end{proof}

The upper bound on the number of isomorphism classes of all quandles, proved by Blackburn in \cite{B}, is $2^{(c+o(1))n^2}$ where $c\approx 1.5566$. For medial quandles, one can easily do better: following the proof of the previous theorem, additionally, we need to bound the number of homomorphism matrices. To do that, an obvious estimate $|\mathrm{Hom}(A,B)|\leq|B|^{\log |A|}$ (since an abelian group $A$ has at most $\log_2|A|$ generators) can be used, which results in the upper bound $2^{(\frac12+o(1))n^2}$ on the number of isomorphism classes medial quandles of order $n$.

While this is a better bound than Blackburn's, we think it is not optimal. Computational results in Table \ref{Fig:count_medial} suggest the following conjecture.

\begin{con}
The number of isomorphism classes of medial quandles of order $n$ is at most $2^{(\frac14+o(1))n^2}$.
\end{con}

Perhaps the same upper bound holds for all quandles, but we lack a computational evidence at this point. The numbers in Table \ref{Fig:count_all} are too small to take into account the fact that the number of non-abelian groups grows much faster than that of abelian groups.

\subsection{Computational results}\label{ssec:computation}

\begin{table}
$$\begin{array}{|r|ccccccccccc|}\hline
n & 1&2&3&4&5&6&7&8&9&10&...\\\hline
\text{all}& 1&1&3&7&22&73&298&1581&11079&&\\
\text{medial}& 1&1&3&6&18&58&251&1410&10311&98577&...\\\hline
\text{involutory}& 1& 1& 3& 5& 13& 41& 142& 665& 4288& 36455& \\
\text{medial involutory}& 1& 1& 3& 4& 11& 33& 121& 597& 4017& 35103& ... \\\hline
\end{array}$$
\caption{The number of quandles of size $n$, up to isomorphism.}
\label{Fig:count_all}
\end{table}

In Table \ref{Fig:count_all}, we compare the numbers of isomorphism classes of all quandles, medial quandles, involutory and involutory medial quandles. McCarron calculated the numbers in the first two rows for $n\leq9$, and in the third row for $n\leq10$, see OEIS sequences A181769, A165200, A178432 \cite{OEIS} (no reference is given there). Earlier, Ho and Nelson \cite{HN} enumerated quandles up to size 8, by an exhaustive search over all permutations that fill the rows of a multiplication table. According to our experiments, the brute force approach, an exhaustive search over all multiplication tables using a SAT-solver, works well up to size 7.

\begin{table}
\begin{small}
$$\begin{array}{|r|rrrrrrrrrrrr|}\hline
n & 1&2&3&4&5&6&7&8&9&10&11&12\\\hline
\text{medial}           & 1& 1& 3& 6& 18& 58& 251& 1410& 10311& 98577& 1246488& 20837439\\
\text{2-reductive}      & 1& 1& 2& 5& 15& 55& 246& 1398& 10301& 98532& 1246479& 20837171\\\hline
\text{medial inv.}      & 1& 1& 3& 4& 11& 33& 121&  597&  4017& 35103&  428081& 6851591 \\
\text{2-red. inv.} 			& 1& 1& 2& 4& 10& 31& 120&  594&  4013& 35092&  428080& 6851545 \\\hline
\end{array}$$
$$\begin{array}{|r|rrrrr|}\hline
n & 13&14&15&16&17\\\hline
\text{medial}           & 466087635& & 563753074951 & &\\
\text{2-reductive}      & 466087624& 13943041873& 563753074915 & 30784745506212 & \\\hline
\text{medial inv.}      & 153025577& 4535779061 & 187380634552 & & 801710433900517\\
\text{2-red. inv.} 			& 153025576& 4535778875 & 187380634539 & 10385121165057 & 801710433900516\\\hline
\end{array}$$
\end{small}
\caption{The number of medial quandles of size $n$, up to isomorphism.}
\label{Fig:count_medial}
\end{table}

Table \ref{Fig:count_medial} displays longer sequences, obtained with our new algorithms based on the affine mesh representation (see Section \ref{ssec:algorithms})\footnote{Our implementation in GAP \cite{gap} can be found at \url{http://www.karlin.mff.cuni.cz/\~{ }stanovsk/quandles}}.
Surprisingly, there are (relatively) very few medial quandles that are not 2-reductive. More detailed information about this class is displayed separately in Table~\ref{Fig:count_medial_nonred}.

\begin{table}
$$\begin{array}{|r|rrrrrrrrrrrrrrrr|}\hline
n & 1&2&3&4&5&6&7&8&9&10&11&12&13&14&15&...\\\hline
\text{non-2-reductive}              &0&0&1&1&3&3&5&12&10&45 &9&268 &11&  &36& \\
\text{reductive, not 2-reductive}   &0&0&0&0&0&2&0& 9& 0&42 &0&260 & 0&  &12& \\
\text{non-reductive}                &0&0&1&1&3&1&5& 3&10& 3 &9&  8 &11&5 &24& \\
\text{all orbits non-trivial latin} &0&0&1&1&3&1&5& 3& 9& 3 &9&  3 &11&5 & 7&... \\
\text{latin}                        &1&0&1&1&3&0&5& 2& 8& 0 &9&  1 &11&0 & 3&...\\\hline
\text{non-2-reductive inv.}              &0&0&1&0&1&2&1&3&4&11&1&46 &1&186 &13& \\      
\text{reductive, not 2-reductive inv.}   &0&0&0&0&0&1&0&3&0&10&0&42 &0&185 & 0& \\      
\text{non-reductive inv.}                &0&0&1&0&1&1&1&0&4& 1&1& 4 &1&  1 &13& \\
\text{all orbits non-trivial latin inv.} &0&0&1&0&1&1&1&0&3& 1&1& 1 &1&  1 & 3&...\\
\text{latin inv.}                        &1&0&1&0&1&0&1&0&2& 0&1& 0 &1&  0 & 1&...\\\hline
\end{array}$$
\caption{The number of medial quandles of size $n$, up to isomorphism.}
\label{Fig:count_medial_nonred}
\end{table}

Latin medial quandles are connected, and thus affine by Corollary
\ref{cor_medial iff affine}. Affine quandles, and latin affine
quandles in particular, were enumerated by Hou \cite{Hou2}. He found
explicit formulas for sizes $p^k$ with $p$ prime and $k=1,2,3,4$,
and it follows from the classification of finite abelian groups that
the function counting the number of affine quandles is
multiplicative. The numbers in Table \ref{Fig:count_medial_nonred}
and in \cite{Hou2} agree. According to Corollary
\ref{cor:2-symmetric}, connected involutory medial quandles arise as
$\Aff(G,-1)$, for certain groups $G$. Such a quandle is latin if and only if $x\mapsto2x$ is
a permutation on $G$; in the finite case, if and only if $|G|$ is
odd. Hence the last row in Table \ref{Fig:count_medial_nonred}
counts the number of abelian groups of odd order.

The class of quandles where all orbits are latin was studied in Section \ref{sec:latin_orbits}. According to Theorem \ref{thm:latin_orbits}, all of them are direct products of a latin quandle $L$ and a projection quandle $P$. Assuming the latin quandle $L$ is non-trivial ($|L|>1$), the product $L\times P$ is non-reductive and the number of such products of size $n$ equals $\sum_{1\neq d\mid n}l(d)$, where $l(d)$ denotes the number of latin medial quandles of size $d$.

\subsection{Enumeration algorithm}\label{ssec:algorithms}

Here we describe our method for enumeration of medial quandles of size $n$ in a given class $\mathcal C$. First, we find all partitions $n=m_1+\dots+m_k$ and consider all $k$-tuples of abelian groups $(B_1,\dots,B_k)$ such that $|B_i|=m_i$, up to reordering and isomorphism of fibres. For the rest of the exposition, consider a fixed tuple $(B_1,\dots,B_k)$ such that $B_i\simeq B_j$ implies $B_i=B_j$, let $A_1,\dots,A_m$ be the list of pairwise non-isomorphic groups that appear in the tuple, and assume $B_1=\dots=B_{n_1}=A_1$, $B_{n_1+1}=\dots=B_{n_1+n_2}=A_2$, and so on.
Denote $X$ the set of all indecomposable affine meshes $(B_i;\,\varphi_{i,j};\,c_{i,j})$ over the set $I=\{1,\dots,k\}$ that result in medial quandles from $\mathcal C$.

To calculate the number of homology classes of meshes from $X$, we use Burnside's orbit counting lemma. Let $G$ be a group acting on the set $X$ such that two meshes $\mathcal A,\mathcal A'\in X$ are homologous if and only if there is $g\in G$ such that $g(\mathcal A)=\mathcal A'$. Let $\sim$ be an equivalence on $G$ such that $g\sim h$ implies $\mathrm{fix}(g)=\mathrm{fix}(h)$, where $\mathrm{fix}(g)$ denotes the number of meshes from $X$ fixed by $g$, and fix a set $R$ of class representatives for $\sim$. Then the number of homology classes equals
$$\frac1{|G|}\cdot\sum_{g\in G}\mathrm{fix}(g)=\frac1{|G|}\cdot\sum_{g\in R}|g/{\sim}|\cdot\mathrm{fix}(g).$$
Remark \ref{rem:iso_thm} suggests that one can always take
$$G=\prod_{i=1}^{m} (A_i\rtimes\aut{A_i})\wr S_{n_i}.$$
For some classes, a simplification is possible.
In theory, we could take $\sim$ the conjugacy equivalence, $g\sim h$ iff $g,h$ are conjugate. In practice, it is hard to handle conjugacy in semidirect products, calculate convenient class representatives and determine class sizes efficiently. We take a complementary approach: we declare a set of representatives $R$ and define a \emph{subconjugacy equivalence over $R$}, i.e., an equivalence $\sim$ such that $g\sim h$ implies $g,h$ are conjugate, and $R$ is a set of class representatives for $\sim$.

First, consider an arbitrary wreath product $H\wr S_n$, and assume $H$ possesses a subconjugacy equivalence $\approx$ over a set $T\subseteq H$. Let $U$ be a set of conjugacy class representatives in $S_n$. We define $$R=\{(g_1,\dots,g_n;\pi)\in H\wr S_n:\ g_1\in T,\ g_2,\dots,g_n\in H,\ \pi\in U\}.$$
For every $\pi\in U$ and every $\sigma\in\pi^{S_n}$, fix $\alpha(\sigma)\in S_n$ such that $\sigma=\pi^{\alpha(\sigma)}$; for $\sigma=\pi$ choose $\alpha(\sigma)=1$.
For every $g\in T$ and every $h\approx g$, fix $\beta(h)\in H$ such that $h=g^{\beta(h)}$; for $h=g$ choose $\beta(h)=1$.
For $(\bar g;\pi)\in R$, $\sigma\in\pi^{S_n}$ and $h\approx g_1$, define
$$(g_1,\dots,g_n;\pi)^{(h,\sigma)}=(g_{\alpha(\sigma)(1)}^{\beta(h)},\dots,g_{\alpha(\sigma)(n)}^{\beta(h)};\sigma)$$
and let $\sim$ be the equivalence with blocks
$$(\bar g;\pi)/{\sim}=\{(\bar g;\pi)^{(h,\sigma)}:\ \sigma\in\pi^{S_n},\ h\approx g_1\}$$
for every $(\bar g;\pi)\in R$.
A straightforward calculation shows that this is a well defined equivalence, i.e., the blocks are pairwise disjoint and cover all $H\wr S_n$.
In fact, $\sim$ is a subconjugacy equivalence over the set $R$, because $(\bar g;\pi)^{(h,\sigma)}$ is a conjugate of $(\bar g;\pi)$ by
$(\beta(h),\dots,\beta(h);\alpha(\sigma))$. It is also easy to calculate that $$|(\bar g;\pi)/{\sim}|=|g_1/{\approx}|\cdot|\pi^{S_n}|,$$
because different pairs $(h,\sigma)$ yield different elements $(\bar g;\pi)^{(h,\sigma)}$.

Now, we return back to the original problem, to determine the equivalence $\sim$ on the group $G$ from Remark \ref{rem:iso_thm}. Since $G$ is a direct product of wreath produts, we can take the product equivalence. It remains to determine a subconjugacy equivalence $\approx$ on $A\rtimes\aut A$.
A similar approach can be used:
fix a set $V$ of conjugacy class representatives in $\aut A$, define $T=\{(a,\varphi):\ a\in A,\ \varphi\in V\}$ and construct a subconjugacy equivalence $\approx$ over $T$ in an analogous way, using the action $(a,\varphi)^\psi=(\gamma(\psi)(a),\psi)$, where $\gamma(\psi)$ satisfies $\varphi^{\gamma(\psi)}=\psi$. In particular, $|(a,\varphi)/{\approx}|=|\varphi^{\aut A}|$.

As indicated in Section \ref{ssec:computation}, there are two essentially different cases to be considered for the enumeration: the class of 2-reductive medial quandles (many models, simple structure), and its complement (few models, complicated structure).

\emph{Non-2-reductive medial quandles.} We take $G$ as in Remark
\ref{rem:iso_thm}, and $\sim$, $\approx$, $R$ as described above. It
remains to explain how to calculate the number $\mathrm{fix}(g)$ of
meshes fixed by $g\in G$. We do it by checking every possible affine
mesh for being fixed. Meshes are constructed by an exhaustive
search: homomorphism matrices first, constant matrices compatible
with each homomorphism matrix next. Partial solutions are being
checked on conditions (M1)-(M4), indecomposability, and a number of
structural properties is used to cut further branches in the search
(Propositions \ref{prop:gcd}, \ref{p:finite_latin_orbit},
\ref{p:latin_orbits} and Lemma \ref{lem:2-reductive} are
particularly helpful). Theorem \ref{thm:reductive} is used to
separate the reductive and non-reductive cases. Results from
Section \ref{sec:latin_orbits} are applied on quandles with latin
orbits, avoiding the exhaustive search in this case.

All numbers in Table \ref{Fig:count_medial_nonred} have been checked by an independent calculation using a different approach. Instead of Burnside's lemma, heuristics are applied to avoid some isomorphic copies in the exhaustive search, and the meshes that are retained are checked upon pairwise isomorphism. For medial quandles that are not 2-reductive, the alternative approach results in similar running times. In the 2-reductive case, it is doomed to fail due to a huge number of meshes.

\emph{2-reductive medial quandles.}
The numbers in Table \ref{Fig:count_medial} indicate that we must avoid storing the meshes. Using Theorems \ref{thm:2-reductive} and \ref{thm:2-reductive_iso}, consider the group
$$G=\prod_{i=1}^m \aut{A_i}\wr S_{n_i}=\left(\prod_{i=1}^k \aut{B_i}\right)\rtimes\left(\prod_{i=1}^m S_{n_i}\right)$$
acting on matrices $(c_{i,j})_{i,j=1..k}$ such that $c_{i,j}\in B_j$, $c_{i,i}=0$ and $B_j=\langle c_{1,j},\dots,c_{k,j}\rangle$ for every $i,j$.
We use $\sim$ and $R$ as described above, and let $\approx$ be the conjugacy equivalence on $\aut{A_i}$ (which is easy to handle computationally).
To calculate the number of fixed meshes, consider the action of a permutation $\pi\in\prod_{i=1}^m S_{n_i}\leq S_k$ on a $k\times k$ table, simultaneously permuting rows and columns, as an oriented graph on a $k\times k$ lattice of vertices. Consider a homology $g=(\bar\psi,\pi)\in G$ and a cycle $c$ in $\pi$. The cycle only permutes coordinates related to a particular group, $A_{j}$. It is sufficient to focus on a single column within the cycle $c$ (call it a $c$-column), since one $c$-column determines the other $c$-columns uniquely. Hence the number of fixed meshes can be calculated as
$$\mathrm{fix}(\bar\psi,\pi)=\prod_{c\text{ cycle in }\pi}\left((\text{\# of $c$-columns fixed by }(\bar\psi,\pi)) - (\text{\# of non-generating $c$-columns})\right).$$
The number of non-generating $c$-columns simply means the number of tuples from $A_{j}^{k-1}$ that do not generate the group $A_{j}$. The number of $c$-columns fixed by $(\bar\psi,\pi)$ counts the following: in how many ways can we supply one $c$-column in a way that the part of the table consisting of all $c$-columns (which are uniquely determined by the given one) is fixed by $(\bar\psi,\pi)$? Looking at the graph of the action of $\pi$, the answer is
$$\prod_{d\text{ cycle in }\pi}\left|\mathrm{fix}_{A_j}\left(\psi_j^{\mathrm{lcm}(|c|,|d|)}\right)\right|^{\ell(c,d)}$$
where $\ell(c,d)$ is the length of the component of the graph related to $c,d$. Clearly, $\ell(c,c)=|c|-1$ and $\ell(c,d)=\gcd(|c|,|d|)$ for $c\neq d$.
We obtained a formula for $\mathrm{fix}(\bar\psi,\pi)$.

\emph{Involutory quandles.} We modify the algorithms described above using Corollary \ref{cor:2-symmetric}. For non-2-reductive quandles, the exhaustive search is pruned by setting $\varphi_{i,i}=2$ for every $i$. In the 2-reductive case, we use the observation that a 2-reductive medial quandle is involutory if and only if its orbit groups have exponent at most two.

\section{A note on congruences}\label{sec:congruences}

This section has a mild universal algebraic flavour, and we refer to \cite{Ber} for any undefined notions from universal algebra.

To proceed further in the theory of medial quandles, we need to learn what congruences and quotients are. Is there a nice description of congruences in the language of affine meshes? What is the mesh for the corresponding quotient? We leave the questions for further study. Partial results for 2-reductive and involutory medial quandles can be found in \cite{RR89, R99b}. Their results were sufficiently strong to characterize subdirectly irreducible algebras in the respective classes, see below. Let us start with simple quandles first.


Finite simple quandles, i.e., finite quandles with no non-trivial congruences, were classified independently in \cite{AG,J82b}. The classification is not easy.
Since the orbit decomposition provides a congruence, simple quandles with more than two elements must be connected, hence, in the medial case, affine.
We cite the characterization of Andruskiewitsch and Gra\~na.

\begin{theorem}\cite[Corollary 3.13]{AG}\label{thm:simple}
A finite medial quandle $Q$ is simple if and only if $Q\simeq\aff{\Z_p^k,M}$ where $p$ is a prime and $M$ is the companion matrix of an irreducible monic polynomial in $\mathbb F_p[x]$.
\end{theorem}

Finite simple medial quandles can also be presented using finite fields: if $b$ is a generator of $\mathbb F_q^*$, then $Q=\aff{\mathbb F_q,b}$ is simple, because $\lmlt Q=\mathbb F_q\rtimes\mathbb F_q^*$ is a doubly transitive group, and all finite simple medial quandle arise this way. 

An algebraic structure is called \emph{subdirectly irreducible} if the intersection of non-trivial congruences is non-trivial. (Subdirectly irreducibles are important since, according to Birkhoff's theorem, every algebra in a variety $\mathcal V$ embeds into a direct product of subdirectly irreducibles in $\mathcal V$, see \cite[Section 3.3]{Ber}). The classification of subdirectly irreducible medial quandles seems to be much harder than that of simple ones, and we leave it as an interesting open problem. Finite subdirectly irreducibles were classified in two special classes of medial quandles, the involutory (2-symmetric) and the 2-reductive ones.

\begin{theorem}\cite[Theorems 3.1 and 4.3]{R99b}\label{thm:sym_si}
A finite involutory medial quandle $Q$ is subdirectly irreducible if and only if $|Q|=2$ or $Q$ is isomorphic to the sum of one of the following affine meshes:
$$((\Z_{p^k}),\,(2),\,(0)),\quad ((\Z_{2^{k}},\Z_{2^{k-1}}),\,\left(\begin{smallmatrix}2&2\\2&2\end{smallmatrix}\right),\,\left(\begin{smallmatrix}0&-1\\1&0\end{smallmatrix}\right)),\quad
((\Z_{2^{k}},\Z_{2^{k-1}},\Z_{2^{k-1}}),\,\left(\begin{smallmatrix}2&2&2\\2&2&2\\2&2&2\end{smallmatrix}\right),\,\left(\begin{smallmatrix}0&-1&0\\1&0&1\\0&-1&0\end{smallmatrix}\right))$$
where $p$ is an odd prime and $k\geq1$.
\end{theorem}

\begin{theorem}\cite[Theorem 3.1]{RR89}\label{thm:red_si}
A finite 2-reductive medial quandle $Q$ is subdirectly irreducible if and only if $|Q|=2$ or $Q$ is isomorphic to the sum of an affine mesh
$$((\Z_{p^k},\Z_1,\dots,\Z_1),\,0,\,(c_{i,j})),$$
where $p^k$ is a prime power, the number $m$ of fibres is at least two, and $c_{2,1},\dots,c_{m,1}\in\Z_{p^k}$ are pairwise different elements such that $\Z_{p^k}=\langle c_{2,1},\dots,c_{m,1}\rangle$.
\end{theorem}

\end{document}